\title{On Gromov's Method of Selecting Heavily Covered Points}

\def\kamitisymb{{\rm a}}
\def\ethsymb{{\rm b}}
\def\snfsymb{{\rm c}}

\author{
{\sc Ji\v{r}\'{\i} Matou\v{s}ek}$^{\kamitisymb, \ethsymb}$
\and
{\sc Uli Wagner}$^{\ethsymb,\, \snfsymb}$
}

\newcommand{\cmt}[1]{\ifhmode\newline\fi{\sf *** \ \ #1 \\}}

\documentclass[11pt,a4paper]{article}
\usepackage[a4paper]{geometry}
\usepackage{latexsym}
\usepackage{amsmath,amssymb,amsthm,bbm}
\usepackage{graphicx}
\usepackage{url}

\newtheorem{theorem}{Theorem}
\newtheorem{definition}[theorem]{Definition}
\newtheorem{prop}[theorem]{Proposition}
\newtheorem{obs}[theorem]{Observation}
\newtheorem{lemma}[theorem]{Lemma}

\theoremstyle{definition}
\newtheorem{example}[theorem]{Example}

\newcommand{\heading}[1]{\vspace{1ex}\par\noindent{\bf #1}}

\makeatletter
\newcommand{\ProofEndBox}{{\ifhmode\unskip\nobreak\hfil\penalty50 \else
          \leavevmode\fi\quad\vadjust{}\nobreak\hfill$\Box$
            \finalhyphendemerits=0 \par}}
\makeatother
\newcommand{\proofend}{\ProofEndBox\smallskip}

\newcommand{\R}{{\mathbb{R}}}
\newcommand{\Z}{{\mathbb{Z}}}
\newcommand{\s}{\mathbb{S}}
\newcommand\eps{\varepsilon}

\newcommand{\lk}{\mathop {\rm lk}\nolimits}

\def\:{\colon}

\long\def\onefigure#1#2{
\begin{figure*}[tbp]
\begin{center}
#1
\end{center}
\caption{#2}
\end{figure*}
}
\def\immediateFigure#1{%
\smallskip\begin{center}#1\end{center}\smallskip }

\newcommand{\labfig}[2]  
{\onefigure{\mbox{\includegraphics{#1}}}{\label{f:#1} #2} }

\newcommand{\labfigw}[3]  
{\onefigure{\mbox{\includegraphics[width=#2]{#1}}}{\label{f:#1} #3}}

\newcommand{\immfig}[1]  
{\immediateFigure{\mbox{\includegraphics{#1}}}}

\newcommand{\immfigw}[2] 
{\immediateFigure{\mbox{\includegraphics[width=#2]{#1}}}}

\newcommand\conjge{{\,\stackrel{??}{\ge}\,}}

\newcommand{\cocyc}[1]{\mathcal{Z}^d}
\newcommand{\cofilling}{cofilling}
\newcommand{\Cofilling}{Cofilling}

\newcommand{\caff}{c}
\newcommand{\ctop}{c^\textup{top}}

\begin{document}

\maketitle
{\renewcommand\thefootnote{\kamitisymb}
\footnotetext{Department of Applied Mathematics and Institute of Theoretical Computer Science (ITI),
Charles University, Malostransk\'{e} n\'{a}m. 25,
118~00~~Praha~1,  Czech Republic}
}
{\renewcommand\thefootnote{\ethsymb}
\footnotetext{Institute of  Theoretical Computer Science,
ETH Z\"urich, 8092 Z\"urich, Switzerland}
}
{\renewcommand\thefootnote{\snfsymb}
\footnotetext{Research supported by the Swiss National Science Foundation (SNF Projects 200021-125309 and 200020-125027)}
}

\begin{abstract} 
A result of Boros and F\"uredi ($d=2$) and
of B\'ar\'any (arbitrary $d$) asserts that for every $d$
there exists $c_d>0$ such that for every $n$-point set  $P\subset \R^d$,
some point of  $\R^d$ is covered by at least
$c_d{n\choose d+1}$ of the $d$-simplices spanned by the points of~$P$.
The largest possible value of $c_d$ has been the subject of ongoing 
research. Recently Gromov improved the existing lower bounds
considerably by introducing a new, topological proof method.

We provide an exposition of the combinatorial component of
Gromov's approach, in terms accessible to combinatorialists
and discrete geometers, and we investigate the limits of his method.

In particular, we give tighter bounds on the \emph{cofilling profiles}
for the $(n-1)$-simplex. These bounds yield a minor improvement
over Gromov's lower bounds on $c_d$ for large $d$, but they also show
that the room for further improvement through the {\cofilling} profiles alone
is quite small. We also prove a slightly better lower bound
for $c_3$ by an approach using an additional structure
besides the {\cofilling} profiles. We formulate a
combinatorial extremal problem whose solution 
might perhaps lead to a tight lower bound for~$c_d$.
\end{abstract}

\section{Introduction}
Let $P\subset \R^2$ be a set of $n$ points in general position (i.e., no three points collinear). Boros and F\"uredi~\cite{BorosFuredi:PlanarSelectionLemma-84} showed that there always exists a point $a\in\R^2$ contained in a positive fraction of all the $\binom{n}{3}$ triangles spanned by $P$, 
namely, in at least $\frac{2}{9}\binom{n}{3} - O(n^2)$ triangles.
(Generally we cannot assume $a\in P$, as the example of points
in convex position shows.)

This result was generalized by B\'ar\'any~\cite{Barany:GeneralizationCarathoedory-1982} to point sets in arbitrary fixed dimension:

\begin{theorem}[\textbf{B\'ar\'any~\cite{Barany:GeneralizationCarathoedory-1982}}]
\label{thm:Barany}
Let $P$ be a set of $n$ points in general position in $\R^d$ (i.e., no $d+1$ or fewer of the points are affinely dependent). Then there exists a point in $\R^d$ that is contained in at least 
$$ \caff_d \cdot \binom{n}{d+1} - O(n^d)$$
$d$-dimensional simplices spanned by the points in $P$, where the constant $\caff_d>0$ (as well as the constant implicit in the $O$-notation) depend only on $d$. 
\end{theorem}
The largest possible value of $\caff_d$ has been the subject of ongoing research. With some abuse of notation, we will henceforth denote by $\caff_d$ the largest possible constant for which Theorem~\ref{thm:Barany} holds true.

\paragraph{Upper bounds.} Bukh, Matou\v{s}ek and Nivasch~\cite{BukhMatousekNivasch:StabbingSimplices-2010} showed that\footnote{In \cite{BukhMatousekNivasch:StabbingSimplices-2010}, a different normalization is used, in the sense that simplices are counted in the form $c_d \cdot n^{d+1} - O(n^d)$. This means that the bounds on $c_d$ given in \cite{BukhMatousekNivasch:StabbingSimplices-2010} have to be multiplied by $(d+1)!$ to match our normalization.}
\begin{equation}
\label{eq:lower-bound-c_d}  
\caff_d \leq  \frac{(d+1)!}{(d+1)^{(d+1)}}\sim \frac{\sqrt{2\pi d}}{e^d} = e^{-\Theta(d)}
\end{equation}
by constructing examples of $n$-point sets $P$ in $\R^d$ for which no point in $\R^d$ is contained more than
$(\tfrac{n}{d+1})^{d+1} -O(n^d)$ many $d$-simplices spanned by $P$. 

\paragraph{Lower Bounds.} The result of Boros and F\"uredi says that $\caff_2 \geq 2/9$, which matches the upper bound in (\ref{eq:lower-bound-c_d}). For general $d$, B\'ar\'any's proof yields 
$$\caff_d \geq \frac{1}{(d+1)^d},$$
which is smaller than (\ref{eq:lower-bound-c_d}) by a factor of $d!$. In \cite{Wagner:Thesis-2003}, the lower bound was improved by a factor of roughly $d$ to 
$$\caff_d \geq \frac{d^2 +1}{(d+1)^{d+1}},$$
but this narrowed the huge gap between upper and lower bounds only slightly. Moreover, Bukh, Matou\v{s}ek and Nivasch showed that the method in \cite{Wagner:Thesis-2003} cannot be pushed farther.
An improvement of the lower bound for $c_3$ by a clever elementary geometric
argument was recently achieved by Basit et al.~\cite{Basit-al}.

We refer to \cite{BukhMatousekNivasch:StabbingSimplices-2010} for a more detailed discussion of the problem and related results.

\paragraph{Gromov's results.}
Recently, Gromov~\cite{Gromov:SingularitiesExpandersTopologyOfMaps2-2010} introduced a new, topological proof method, which improves on the previous
 lower bounds considerably and which, moreover, applies to a more 
general setting, described next.

An $n$-point set $P \subseteq \R^d$ determines an affine map $T$ from the $(n-1)$-dimensional simplex $\Delta^{n-1} \subseteq \R^{n-1}$ to $\R^d$ as follows. Label the vertices of $\Delta^{n-1}$ by $V=\{v_1,\ldots, v_n\}$, and the points as $P=\{p_1,\ldots, p_n\}$. Then $T$ is given by mapping $v_i$ to $p_i$, $1\leq i \leq n$, and by interpolating linearly on the faces $\Delta^{n-1}$. 

Thus, B\'ar\'any's result can be restated by saying that for any \emph{affine} map $T \colon \Delta^{n-1} \rightarrow \R^d$, there exists a point in $\R^d$ that is contained in the $\psi$-images of at least $\caff_d \cdot \binom{n}{d+1} -O(n^d)$ many $d$-dimensional faces of $\Delta^{n-1}$.

Gromov shows that, more generally\footnote{In fact, Gromov's approach is still much more general than this and applies to continuous maps from arbitrary finite simplicial complexes $X$ to arbitrary $d$-dimensional manifolds $Y$. The method yields lower bounds for the maximum number of $d$-simplices of $X$ whose images share a common point as long as $X$ has certain \emph{expansion properties}. This will be briefly explained in Section~\ref{subsec:abstract} below.}, the following is true:

\begin{theorem}[\textbf{\cite{Gromov:SingularitiesExpandersTopologyOfMaps2-2010}}]
\label{thm:Gromov-Simplex-Selection}
For an arbitrary \emph{continuous} map $T \colon \Delta^{n-1} \rightarrow \R^d$, there exists a point in $\R^d$ that is contained in the $T$-images of at least $\ctop_d \cdot \binom{n}{d+1}-O(n^d)$  many $d$-faces of $\Delta^{n-1}$, where $\ctop_d>0$ is a constant that depends only on $d$.
\end{theorem}
By the same abuse of notation as before, $\ctop_d$ will also henceforth denote the largest constant for which Theorem~\ref{thm:Gromov-Simplex-Selection} holds. In particular, $\caff_d \geq \ctop_d$.

Gromov's method gives
\begin{equation}
\label{eq:c_d-top}
\caff_d \geq \ctop_d \geq \frac{2d}{(d+1)! (d+1)}
\end{equation}

For $d=2$, this yields the tight bound $\ctop_2 = \caff_2 =2/9$. 
For general $d$, Gromov's result improves
on the earlier bounds by a factor exponential in $d$, but it
is still of order $e^{-\Theta(d\log d)}$ 
and thus far from the upper bound.

One of the goals of this paper is to provide an exposition of the combinatorial component of Gromov's approach, in terms accessible to combinatorialists and discrete geometers.
\medskip

Very recently, after a preliminary version of this paper was written and circulated, Karasev~\cite{Karasev:Gromov-2010} found a very short and elegant proof of Gromov's bound (\ref{eq:c_d-top}) for \emph{affine} maps. 

%
%

Karasev's proof, which he himself describes as a ``decoded and refined'' version of Gromov's proof, combines probabilistic and topological arguments, but he avoids the heavy topological machinery applied in Gromov's proof and only uses the elementary notion of the degree of a piecewise smooth map between spheres.

Karasev's argument can be modified and extented so that it covers the case of arbitrary continuous maps into $\R^d$ (but not yet into general $d$-dimensional manifolds). Furthermore, the combinatorial and the topological aspects of the argument, which are intertwined in Karasev's proof, can be split into two independent parts. In this way, any combinatorial improvement on the cofilling profiles or on the pagoda problem introduced in Section~\ref{sec:c3} immediately imples improved bounds for Theorem~\ref{thm:Gromov-Simplex-Selection} also via this simpler topological route. This will be discussed in more detail in a separate note.

\paragraph{Coboundaries and {\cofilling} profiles.} We need two basic notions, \emph{coboundaries} and \emph{cofilling profiles}, which have their roots in cohomology but which can be defined in  elementary and purely combinatorial terms.

Let $V$ be a fixed set of $n$ elements, w.l.o.g., $V=[n]:=\{1,2,\ldots, n\}$. We will always assume that $n$ is sufficiently large. 

In topological terms, we think of the $(n-1)$-dimensional simplex $\Delta^{n-1}$ as a combinatorial object (an abstract simplicial complex), namely, as the system of all subsets of $V$. Thus, given a subset $f \subseteq V$, we will also sometimes refer to $f$ as a \emph{face} of $\Delta^{n-1}$, and the \emph{dimension} of a face is defined as $\dim f:= |f|-1$. 

Let $E\subseteq {V\choose d}$ be a system of (unordered) $d$-tuples,
or in other words,  of $(d-1)$-dimensional faces. 
We write $\|E\|:= |E|/{n\choose d}$ for the \emph{normalized
size} of $E$; one can also interpret $\|E\|$ as
 the probability that a random $d$-tuple lies in~$E$. 
The notation $\|E\|$ implicitly refers to
$d$ and $n$, which have to be understood from the context.

The \emph{coboundary} $\delta E$ is the system of those $(d+1)$-tuples
in $f\in {V\choose d+1}$ that contain an odd number of $e\in E$. 
(For $d=2$, this notion and some of the following considerations
are related to \emph{Seidel switching} and \emph{two-graphs},
which are notions studied in combinatorics---see the end of
Section~\ref{s:seidel}
for an explanation and references.)

We emphasize that
$\delta E$ also depends on the ground set $V$, and sometimes we may write $\delta_V E$ 
instead of $\delta E$ to avoid ambiguities.

Many different $E$'s may have the same coboundary.
We call $E$ \emph{minimal} if $\|E\|\le \|E'\|$ for every $E'$ 
with $\delta E'=\delta E$. 

We define the \emph{{\cofilling} profile}%
\footnote{Gromov uses the notation $\|(\partial^{d-1})^{-1}_{\rm fil}\|(\beta)$ for
what we would write as $\varphi^{-1}_d(\beta)/\beta$.
Actually, he does not take the $\liminf$, which we use in order  
to avoid  dealing with small values of~$n$.} 
as follows:
$$
\varphi_d(\alpha):=\liminf_{|V|=n\to\infty}\min \{\|\delta E\|: E\in {\textstyle \binom{V}{d}} \mbox{ minimal}, \|E\|\ge \alpha\}.
$$

Equivalently, one can also view this notion as follows. Suppose we are given a system $F \in \binom{V}{d+1}$, and we are guaranteed that $F$ is a coboundary, i.e., that $F=\delta E$ for some $E$. Now we want to know the smallest possible
(normalized) size $\|E\|$ of an $E$ with $F=\delta E$, 
as a function of $\|F\|$. It suffices to consider minimal $E$'s, and 
$\varphi(\alpha) \geq \beta$ means that if we are forced to take 
$\| E\| \geq \alpha$, then we must have $\|F\| \geq \beta$. 

We also remark that there is no minimal $E$ with $\|E\| > 1/2$ (see Section~\ref{sec:basics}), so formally, $\varphi_d(\alpha)=\infty$  for $\alpha >1/2$
(since we take the minimum over an empty set).

As a warm-up, let us consider case $d=1$: In this case, we can view an $S\in \binom{V}{1}$ simply as a subset of $S\subseteq V$, and $\partial S$ 
is the set of edges of the complete bipartite graph with
color classes $S$ and $V\setminus S$; in graph theory, 
one also speaks of the \emph{edge cut} determined by $S$
in the complete graph on $V$. 
\immfig{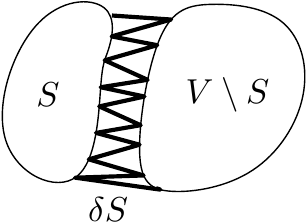}
The minimality of $S$
simply means that  $|S| \leq n/2$. It follows that
$\varphi_1(\alpha)=2\alpha(1-\alpha)$, $0\le\alpha\le\frac12$.

For general $d$, the following basic bound for $\varphi_d$ was observed by Gromov, and independently by Linial, Meshulam, and Wallach \cite{LinialMeshulam:HomologicalConnectivityRandom2Complexes-2006,MeshulamWallach:HomologicalConnectivityRandomComplexes-2009} (and maybe by others). In our terminology, it can be stated as follows:

\begin{lemma}[\textbf{Basic {\Cofilling} Bound}] 
\label{lem:basic}
For every $d\geq 1$ and all $\alpha\in [0,1]$,
$$
\varphi_d(\alpha)\ge \alpha.
$$
\end{lemma}

We will recall a simple combinatorial proof of this bound, along with other basic properties of the coboundary operator, in Section~\ref{sec:basics}. A simple example shows that the basic bound is attained with equality for
$\alpha=\tfrac{(d+1)!}{(d+1)^{(d+1)}}\approx e^{-(d+1)}$, 
but for smaller $\alpha$, improvements are possible,
as we will discuss later.

In Section~\ref{sec:topological-outline}, we present an outline of topological part of Gromov's argument, which yields the following general lower bound.
\begin{prop}[\textbf{\cite{Gromov:SingularitiesExpandersTopologyOfMaps2-2010}}] 
\label{prop:GromovFillingBaranyConstant}
For every $d\geq 1$,
$$
\ctop_d\ge \varphi_{d}(\tfrac12 \varphi_{d-1}(\tfrac13 \varphi_{d-2}(\ldots \tfrac 1d \varphi_1(\tfrac1{d+1})\ldots))).
$$
\end{prop}
Theorem~\ref{thm:Gromov-Simplex-Selection} follows from this proposition 
by using $\varphi_1(\alpha)=2\alpha(1-\alpha)$ and the basic bound for 
all $d\ge 2$. 

Better bounds on the $\ctop_d$ would immediately follow from
Proposition~\ref{prop:GromovFillingBaranyConstant} if
one could improve on the basic cofilling bound 
 in an appropriate range of $\alpha$'s;
this is a purely combinatorial question (and a quite nice one, in
our opinion).

We establish the following lower bounds for $\varphi_2$ and $\varphi_3$:

\begin{theorem}
\label{thm:cofilling-2} For $d=2$ and all $\alpha\le\frac14$,
 we have the lower bound
$$
\varphi_2(\alpha) \ge \tfrac34\left(1-\sqrt{1-4\alpha}\right)(1-4\alpha)=
\tfrac 32\alpha-\tfrac92\alpha^2-3\alpha^3-O(\alpha^4).
$$
\end{theorem}

Fig.~\ref{f:f2bounds} shows a plot of this lower bound.

\labfigw{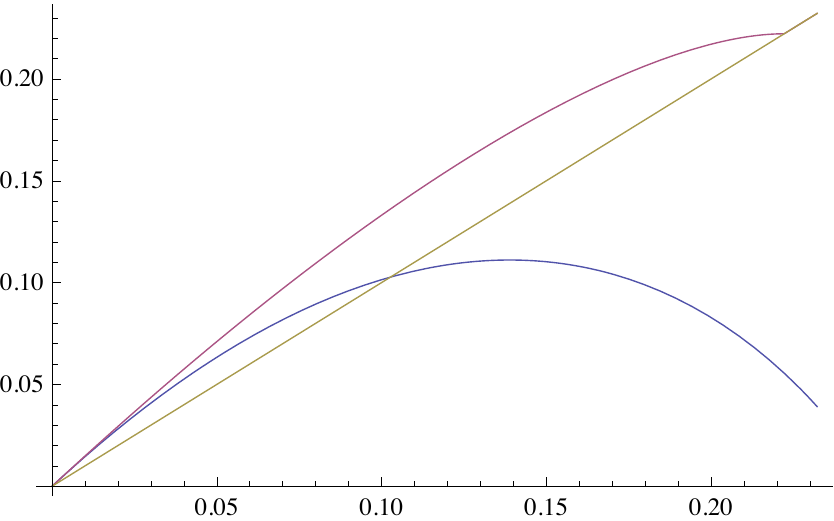}{12cm}{The  lower and upper bounds on $\varphi_2(\alpha)$:
the straight line is the basic bound,
the top one the upper bound (Proposition~\ref{p:ubbb}),
and the bottom (most curved) is the
lower bound (Theorem~\ref{thm:cofilling-2}).}

\begin{theorem}
\label{t:3ub}
For $d=3$ and $\alpha$ sufficiently small,
$$
\varphi_3(\alpha)\ge \frac 43\alpha- O(\alpha^2)
$$
(with a constant that could be made explicit).
\end{theorem}

These theorems will be proved in Sections~\ref{sec:cofilling-2} and \ref{sec:cofilling-ge2}, respectively. 

They do not improve on Gromov's lower bounds
for $c_3$, for example, since they do not beat the basic bound
for the values of $\alpha$ needed in 
Proposition~\ref{prop:GromovFillingBaranyConstant} for $d$ small.
However, they do apply if we take $d$ sufficiently large
in Proposition~\ref{prop:GromovFillingBaranyConstant},
and so they at least show that Gromov's lower bound
on $\ctop_d$ is not tight for large~$d$.

After the research reported in this paper was completed, Kr\'a{l'} et
  al.~\cite{KMS} proved the lower bound $\varphi_2(\alpha)\ge
  \frac97\alpha(1-\alpha)$. This is better than the bound of
  Theorem 5 for $\alpha$ larger than approximately $0.0626$,
  and it does improve on Gromov's lower bound for~$c_3$.

The bounds in Theorems~\ref{thm:cofilling-2}
and~\ref{t:3ub} may look like only minor improvements over the 
basic bound, but it turns out that they have the right order
of magnitude for $\alpha$ tending to~$0$. Indeed, 
we have the following upper bound on the $\varphi_d$'s.

\begin{prop}\label{p:ubbb}
For all $d\geq 1$ and $\alpha \leq \frac{1}{d+1}$, 
let $\sigma\in [0,1)$ be the smallest positive number with 
$\alpha=d!\sigma(\frac{1-\sigma}d)^{d-1}$. Then
$$\varphi_d(\alpha) \leq \tfrac{d+1}d\alpha(1-\sigma),$$
and consequently,   $\varphi_d(\alpha)\le \frac{d+1}d\alpha$.
\end{prop}

These bounds are plotted in Fig.~\ref{f:upperboundplots}.

\labfigw{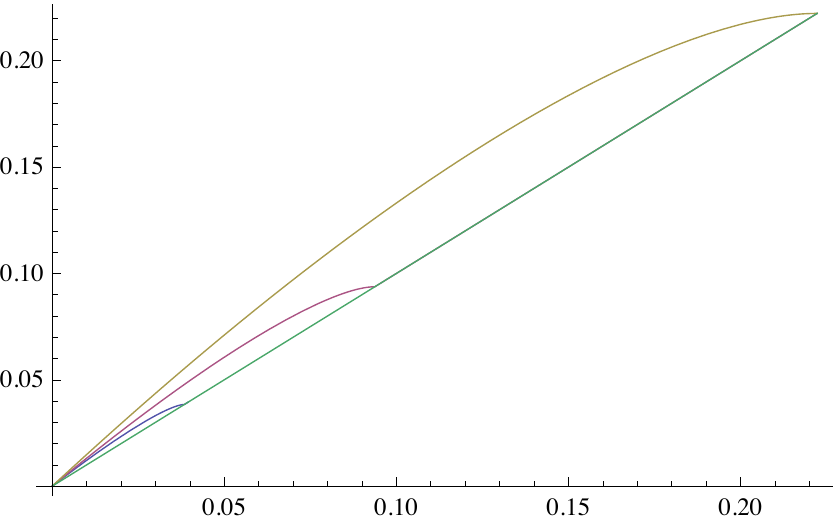}{10cm}{
Plots of the upper bounds from Proposition~\ref{p:ubbb}
for $d=2,3,4$, together with the basic bound.}

Proposition~\ref{p:ubbb} follows from a simple example, whose special
case with $\alpha=\frac1{d+1}$ has already been noted in 
\cite{Gromov:SingularitiesExpandersTopologyOfMaps2-2010}
and in \cite{LinialMeshulam:HomologicalConnectivityRandom2Complexes-2006,MeshulamWallach:HomologicalConnectivityRandomComplexes-2009}.
We present the example and its analysis in Section~\ref{s:multipart}.

We conjecture that the bound in Proposition~\ref{p:ubbb}
is the truth, and moreover, that the example mentioned above
is essentially the only possible extremal example.
However, a proof may be challenging even for the $d=2$ case.
On the other hand, we believe that a suitable extension of 
the proof of Theorem~\ref{t:3ub} may provide a bound of the form
$\varphi_d(\alpha)\ge \frac{d+1}d\alpha-o(\alpha)$ as $\alpha\to 0$.
At present it seems that such an extension would
 be highly technical and complicated.

In view of the upper bound $\varphi_d(\alpha)\le
\frac{d+1}d\alpha$, Gromov's lower bound on $c_d$
cannot be improved by more than a factor of roughly $d$
using Proposition~\ref{prop:GromovFillingBaranyConstant}
alone. 

In Section~\ref{sec:c3},
we introduce a somewhat different approach, which goes beyond 
Proposition~\ref{prop:GromovFillingBaranyConstant} and uses
additional combinatorial structure, and we show that
it can provide a slightly better lower bound for $c_3$.
We formulate a combinatorial extremal problem whose solution might perhaps lead to a tight lower bound for the~$\ctop_d$.

\section{Basics}
\label{sec:basics}
\label{s:multipart}
\label{s:seidel}

\heading{Linearity of the coboundary.}
For systems $E,E' \in\binom{V}{d}$ of $d$-tuples, $E+E'$ means the symmetric difference.

The coboundary operator is (well known and easily checked to be)
\emph{linear} with respect to this operation, i.e., 
$$\delta(E+E')=\delta E+\delta E',$$ 
and we have
$$\delta\delta E=0,$$ 
where $0$ means the empty system of $(d+2)$-tuples.

\heading{Cochains, coboundaries, and cocycles.}
A system $E\in \binom{V}{d}$ of $d$-tuples is also sometimes called a $(d-1)$-dimensional \emph{cochain}\footnote{Strictly speaking, a cochain with $\Z_2$-coefficients in the simplex $\Delta^{n-1}$.}, or simply $(d-1)$-cochain. Later on, when working with systems of tuples of various arities, it will be convenient to use this terminology. A cochain $E$ is called a \emph{cocycle} if $\delta E=0$, and it is a \emph{coboundary} if it can be written as $E=\delta D$ for some $(d-2)$-chain $D$.

In algebraic terms, a $(d-1)$-cochain $E\subseteq \binom{V}{d}$ can be identified with a $0/1$-vector indexed by the elements of $\binom{V}{d}$, which we can interpret as an element of the vector space $\Z_2^{\binom{V}{d}}$ over the $2$-element field $\Z_2$, and the symmetric difference $+$ corresponds to the usual addition is this vector space. The coboundary operators are linear maps between these spaces (mapping $d-1$-cochains to $d$-cochains). The $(d-1)$-cocycles are precisely the elements of the kernel of this linear map, and the $d$-coboundaries are the elements of the image. The property $\delta\delta =0$ is usually called the \emph{chain complex property}.

\paragraph{Minimality.} As we have seen, every coboundary is also a cocycle, i.e., if $F=\delta E \in \binom{V}{d+1}$, then $\delta F=0$. 

In our setting, this is a complete characterization, i.e., $F \in \binom{V}{d+1}$ is a coboundary if and only if $\delta F=0$. Topologically speaking, this is because the $(n-1)$-simplex has zero cohomology. 

We stress that there is one exceptional case, namely,
$d=0$. There are two $0$-cocycles $F\in \binom{V}{1}$, namely, $F=V$, 
the set of all vertices, and $F=0$, the empty set of vertices. 
However, \emph{by definition}, $0$ is considered to be a coboundary, 
but $V$ is not. In topological terms, this is because we are working 
with ordinary, \emph{non-reduced} cohomology.\footnote{We remark that from a combinatorial point of view, it may be more natural also to consider the subsets of $\binom{V}{0}$. There are two such subsets, namely $\{\varepsilon \}$ (the singleton set containing the unique $0$-tuple $\varepsilon$ of vertices of $V$) and $0$ (the empty set of $0$-tuples). For reduced cohomology, one defines $\delta \{\varepsilon \} = V$, so that the exceptional case in the characterization of cocycles disappears, but for the topological theorem that Gromov applies, it is important to work with unreduced cohomology.}  
This will be formally important in Section~\ref{subsec:simplicialsets}.

Because of the equivalence of cocycles and coboundaries, for $d-1 \geq 1$, 
minimality of $E$ can be equivalently characterized 
as follows: $E$ is minimal if $\|E\|\le \|E+\delta D\|$ for
every $D\subseteq {V\choose d-1}$. 
Thus, minimality means that $E$ contains at most half
of the $d$-tuples from each system of the form $\delta D$.
Consequently, if $E$ is minimal and $E'\subseteq E$,
then $E'$ is minimal as well. 

It also follows from the alternative characterization of minimality that every minimal $E$ satisfies $\|E\| \leq 1/2$. 
To see this, let $D=\{x\}$ be a singleton set consisting of a single $(d-1)$-tuple $x$. Then $|E \cap \delta D| \leq 1/2 | \delta D|$ means that $x$ is incident to at most $|\delta D|/2=\frac{n-d+1}{2}$ many $d$-tuples $e \in E$. Summing up over all $x \in \binom{V}{d-1}$, we get $d |E| \leq \binom{n}{d-1} (n-d+1)/2$, and hence $|E| \leq \binom{n}{d}/2$. 

As was remarked above, minimality refers to a fixed vertex set $V$. 
If a minimal $E \subseteq \binom{V}{d}$ happens to be contained
in $\binom{W}{d}$ then, considered as a set of subsets of $W$, it need
no longer be minimal. For example, suppose we partition $V$ into three vertex sets $V_1, V_2, V_3$ of size $n/3$ each and let $E=\{ \{v_1,v_2\}: v_1\in V_1, v_2 \in V_2\}$, the (edge set of the) complete bipartite graph between $V_1$ and $V_2$. Then it is not hard to check that $E$ is minimal as a subset of $\binom{V}{2}$, but not as a subset of $\binom{V_1 \cup V_2}{2}$.

\paragraph{Links of vertices.}
For a vertex $v\in V$, let us write $E_v:=\{e\in E: v\in e\}$.
The \emph{link} of $v$ in $E$ is the system of $(d-1)$-tuples
$\lk(v,E):=\{e\setminus \{v\}: e\in E_v\}$.

It is easy to check the following formula for the coboundary of
the link:
\begin{equation}\label{e:link-cobo}
\delta\lk(v,E) = E_v + \lk(v,\delta E_v)
\end{equation}
(the $+$ on the right-hand side is actually a disjoint union in this case).

\paragraph{The basic bound for the filling profile.} 
We now recall a combinatorial proof of the basic lower bound for $\varphi_d$.

\begin{proof}[Proof of Lemma~\ref{lem:basic}]
 Let $F\subseteq{V\choose d+1}$  be a coboundary, and let $\beta:=\|F\|$.
We define the \emph{normalized degree} of $v$ as 
$\|\lk(v,F)\|=|\lk(v,F)|/{n\choose d}$. 

A simple counting
shows that the average normalized degree of a vertex equals $\beta=\|F\|$.
In particular, there exists a vertex $v$ of normalized degree at most $\beta$;
so we fix such a $v$ and 
we set $E:=\lk(v,F)$. We will check that $\delta E=F$.

Let $F_{\setminus v}:=F\setminus F_v$. Since $F$ is a coboundary,
we have $\delta F=0$, and thus $\delta F_v=\delta F_{\setminus v}$.

Let us consider an arbitrary $(d+1)$-tuple $f$ and distinguish two cases.
If $v\in F$, then it is easily seen that $f\in \delta E$
is equivalent to $f\in F$. 

Next, let $v\not\in f$. Assuming $f\in\delta E$, we have
$f^+:=f\cup\{v\}\in \delta F_v$. Using $\delta F_v=F_{\setminus v}$,
we get $f^+\in \delta F_{\setminus v}$. Since the sets of $F_{\setminus v}$
avoid $v$, there is only one set of $F'$ that may be contained
in $f^+$, namely~$f$. So $f\in F$. This argument can be reversed,
showing that $f\in F$ implies $f\in\delta E$.
\end{proof}


\paragraph{The upper bound example. }
Here we prove Proposition~\ref{p:ubbb}, the upper bound
on the {\cofilling} profile $\varphi_d$.

Let $\sigma \in [0,1]$ be a parameter. We partition the vertex set $V$ 
into $V_1,\ldots,V_{d+1}$,
where $|V_1|=\sigma n$ and the remaining vertices are divided evenly, i.e.,
$|V_i|= \frac{1-\sigma}d n$, $i>1$ (we ignore divisibility issues).
Let $E$ consist of all $d$-tuples that use exactly
one point from each of $V_1,\ldots,V_d$; see Fig.~\ref{fig:UB-cofilling}.
We have $\|E\|=:\alpha=d!\sigma(\frac{1-\sigma}d)^{d-1}$.

\begin{figure}[tb]
\begin{center}
\includegraphics[scale=1]{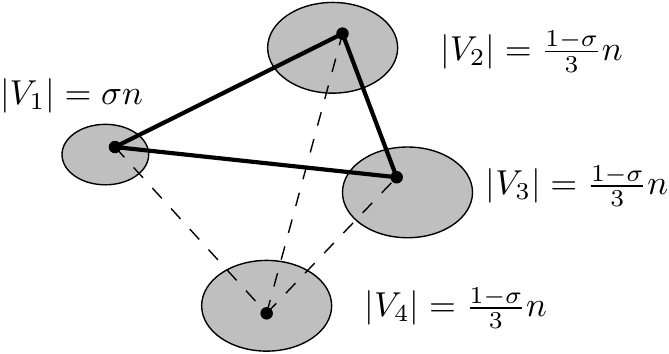}
\caption{The upper bound example in the case $d=3$. The solid triangle depicts a typical element of $E$, and the dashed segments complete it to a tetrahedron that is a typical element of $\delta E$.\label{fig:UB-cofilling}}
\end{center}
\end{figure}

Then $F:=\delta E$ is the complete $(d+1)$-partite system on
$V_1,\ldots,V_{d+1}$, and $\|\delta E\|=\frac{d+1}d\alpha(1-\sigma)$.
This matches the quantitative bounds in Proposition~\ref{p:ubbb},
and it remains to check the minimality of $E$,
which is easy  (and stated by Gromov and by Meshulam et al.\ without
proof): Let us consider some $E'$ with $F=\delta E'$. 
Every $f\in F$ contains at least one
$e\in E'$, while every $e\in E'$ is contained in at most
$M:=\max_i{|V_i|}$ many $f\in F$. So $|E'|\ge |F|/M=|E|$.
\proofend


\paragraph{On Seidel's switching. } In combinatorics,
a \emph{two-graph} is a set $F\subseteq {V\choose 3}$ of triples
with $\delta F=0$,i.e., a cocycle in our terminology).
As we have mentioned, this is equivalent to $F$ being a coboundary,
i.e., to the existence of some $E\subseteq
{V\choose 2}$ with $F=\delta E$. The system of all possible
$E'$ with $\delta E'=F$ is called the \emph{Seidel switching class}
of~$E$.

Two-graphs and Seidel's switching were introduced by Van Lint and Seidel
\cite{vL-S} and further studied by many authors,
because of their connections with equiangular lines,
 strongly regular graphs, and interesting finite groups, for example
(for surveys see, e.g., Seidel and Taylor
\cite{SeidelTaylor} or Hage \cite{Hage-thesis}).  

Numerous authors investigated the computational complexity
of various problems related to Seidel's switching
(we refer to \cite{Jeli-al} for citations). 
For us, the following result is of particular interest:
It is NP-complete to decide if a given $E\subseteq {V\choose 2}$
is minimal (in its Seidel switching class), as was proved
by Jel\'{\i}nkov\'a et al.~\cite{Jeli-al}. Their reduction
produces only $E$'s with $\|E\|\ge\frac 12$ and
$\|E'\|\ge\frac14$ for all $E'$ in the same switching class;
however, recently Jel\'{\i}nek (private communication, September 2010)
was able to modify the reduction, showing that
the problem remains NP-complete even if we restrict to
only to $E$'s with $\|E\|\le c$, for every fixed $c>0$.
This shows that minimal sets have a complicated structure, and
one cannot expect to find a reasonable characterization.

\section{An Outline of Gromov's Topological Approach}
\label{sec:topological-outline}
In this section, we give a rather informal and elementary outline of the topological part of Gromov's approach
(Sections~2.2, 2.4, 2.5, and 2.6 of \cite{Gromov:SingularitiesExpandersTopologyOfMaps2-2010}). This outline
is not directly related to the new (combinatorial) results of the present paper.

We strive to keep the discussion as elementary as possible for as long as possible. For this reason, we restrict ourselves to the most basic (affine) setting of finite set $P$ of $n$ points in general position in $\R^d$, which allows us to describe most steps in the argument in simple geometric terms.


As remarked above, Gromov's method applies in much more general situations. In Section~\ref{subsec:abstract}, we briefly discuss the more general setting and also include some remarks as to how our elementary discussion 
would be formulated in more standard topological terms.


We begin with an outline of our outline, by means of an example.
\begin{example}
Let $P=\{p_1,\ldots, p_5\}$ be the set of five points in $\R^2$ depicted by bold dots in 
Figure~\ref{fig:PlanarAffineExample}, and let $V=[5]:=\{1,2,3,4,5\}$.
\begin{figure}[tb]
\centerline{\includegraphics[scale=0.7]{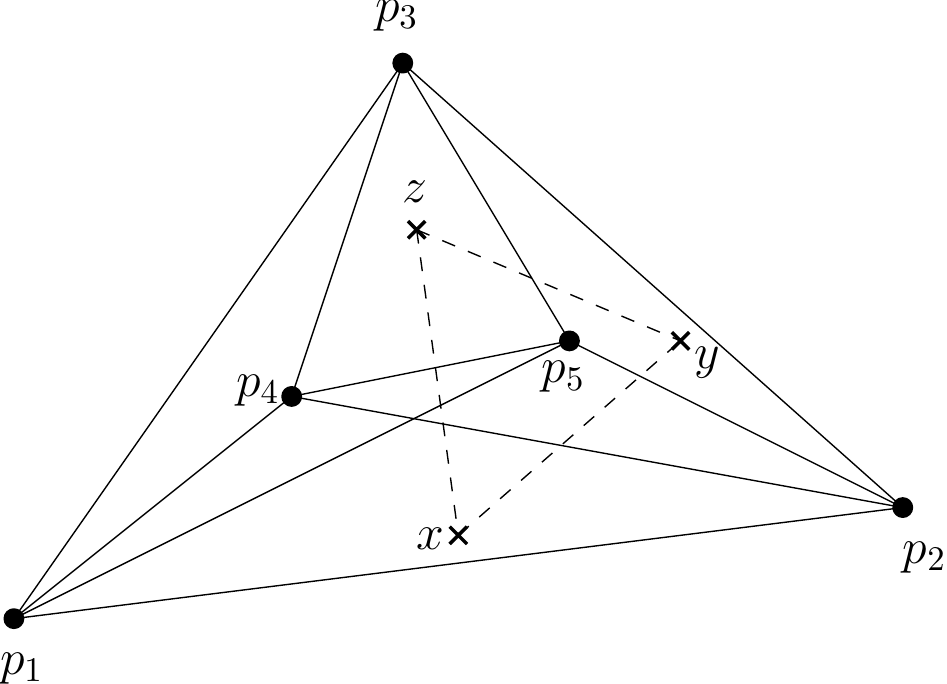}}
\caption{A set of five labeled points in general position in $\R^2$ (the image of $\Delta^4$ under an affine map).\label{fig:PlanarAffineExample}}
\end{figure}

Consider the three points $x,y,z$ marked by crosses in the picture. These three points are in general position w.r.t. $P$, in the sense that they do not lie on any line segment spanned by $P$, and no point of $P$ lies on any of the line segments spanned by $x,y,z$.

Let $F_x=\{ \{1,2,3\}, \{1,2,4\}, \{1,2,5\}\}$ be the set of all triples $\{i,j,k\} \in \binom{V}{3}$ such that $x$ lies in the triangle $p_i p_j p_k$, and let $F_y=\{\{1,2,3\},\{2,3,4\},\{2,3,5\}\}$ and $F_z=\{\{1,2,3\}, \{1,3,5\},$ $\{2,3,4\},\{3,4,5\}\}$ be defined analogously as the (index sets of) triangles containing $y$ and $z$, respectively.

Let $F_{xy}=\{ \{2,4\}, \{2,5\}\}$ be the set of pairs $\{i,j\} \in \binom{V}{2}$ such that the line segment $p_i p_j$ intersects the line segment $xy$, and let $F_{yz}=\{\{3,5\}\}$ and $F_{xz}=\{ \{1,5\}, \{2,4\}, \{4,5\}\}$ be defined analogously.

Finally, let $F_{xyz}= \{5\}$ be the set of indices $i \in V$ such that $p_i$ lies in the triangle $xyz$ (here, we identify elements $i \in V$ with singleton sets $\{i\} \subseteq V$ to simplify notation).

The basic observation is that these sets satisfy the relations
$$
\delta F_x = \delta F_y = \delta F_z =0,
$$
$$
F_x + F_y = \delta F_{xy},\quad F_y + F_z =\delta F_{yz}, \quad F_x+ F_z =\delta F_{xz},
$$
and
$$
F_{xy} + F_{yz} + F_{xz} = \delta F_{xyz}.
$$ 

It is straightforward to verify this in the specific example at hand, but it may in fact be easier---and an instructive 
exercise---to check that these relations are not a coincidence, but hold in general for any finite set $P \subseteq \R^2$ and any triple of points $x,y,z$, 
assuming only general position. Moreover,
similar facts hold in $\R^d$.
We discuss a proof of the general case below.
\end{example}

Informally speaking and in very general terms, the structure of Gromov's topological approach can be summarized as follows. Each step of the argument will be discussed in more detail in a separate subsection below. We fix $d\geq 1$ (the target dimension) and $V=[n]$ (the vertex set of the $(n-1)$-simplex $\Delta^{n-1}$).
\begin{enumerate}
\item We define a topological space $\cocyc{d}=\cocyc{d}(\Delta^{n-1})$, the space of $d$-dimensional cocycles (of the $(n-1)$-simplex).

\indent {\small This space is a \emph{simplicial set}, i.e., a space built of vertices, edges, triangles, and higher-dimensional simplices like a simplicial complex, but simplices are allowed to be glued to each other and to themselves in more general ways (in a first approximation, simplicial sets can be thought of as higher-dimensional analogues of multigraphs with loops, while simplicial complexes are higher-dimensional analogues of simple graphs).

The vertices of $\cocyc{d}$ are the $d$-dimensional cocycles $F\subseteq \binom{V}{d+1}$. The edges of $\cocyc{d}$ correspond to relations of the form $F_1 + F_2 = \delta F_{12}$, where $F_{12}\subseteq \binom{V}{d}$. The triangles of $\cocyc{d}$ correspond to a triple of $d$-dimensional cocycles $F_1$, $F_2$, $F_3$, edges between them, and a relation of the form $F_{12} + F_{23} + F_{31} = \delta F_{123}$, where $F_{123} \subseteq \binom{V}{d-1}$, etc. }

We stress that $\cocyc{d}$ depends only on $n$ and $d$ and is defined purely combinatorially. Moreover, as a combinatorial object, it is huge. For instance, the number of vertices of $\cocyc{d}$ ($d$-dimensional cocycles of the simplex) is $2^{\binom{n-d}{d}}$.

\item With every labeled $n$-point set $P\subseteq \R^d$, we associate a particular subspace\footnote{Formally, it would be more precise to regard $\mathcal{W}$ as a $d$-dimensional $\Z_2$-homology cycle, i.e., as a formal $\Z_2$-linear combination of $d$-simplices in $\cocyc{d}$ such that $\partial\mathcal{W}=0$. However, since we are working with $\Z_2$-coefficients, we can simply think of $\mathcal{W}$ as a subspace, given as the union of those $d$-simplices that appear an odd number of times in the formal sum.} $\mathcal{W} =\mathcal{W}(P) \subseteq \cocyc{d}$.

{\small A concrete way of doing this is to choose a triangulation $\mathcal{T}$ of $\R^d$ that is in general position w.r.t. $P$ (i.e., no $k$-dimensional simplex of $\mathcal{T}$ intersects any $\ell$-dimensional simplex spanned by $P$ if $k+\ell< d$). With every vertex $x$ of $\mathcal{T}$ we associate the set 
$$F_x:=\{ f=\{i_0,i_1,\ldots ,i_d\} \in  \textstyle{\binom{V}{d+1}}: x \in p_{i_0} p_{i_1}\ldots p_{i_d}\}$$ 
of (indices of) $d$-simplices spanned by $P$ that contain $x$. As indicated by the example, each such $F_x$ is a cocycle, i.e., a vertex of $\cocyc{d}$ (but not all vertices of $\cocyc{d}$ may be of this special form). 

With every edge $xy$ of $\mathcal{T}$, we associate the set $F_{xy}$ of $d$-tuples from $V$ such that the corresponding points of $P$ span a $(d-1)$-simplex that intersects $xy$. As in the example, we get the relation $F_x + F_y =\delta F_{xy}$, and hence an edge of $\cocyc{d}$. 

Similarly, each $k$-dimensional simplex of the triangulation gives rise to a $k$-simplex of $\cocyc{d}$ (but not all $k$-simplices in $\cocyc{d}$ may be of this form). We define the subspace $\mathcal{W}$ to consist of all simplices of $\cocyc{d}$ that are obtained from an odd number of simplices of $\mathcal{T}$ using this construction. (In principle, different simplices of the triangulation $\mathcal{T}$ may tield the same $k$-simplex in $\cocyc{d}$.)}

\item It follows from a theorem in algebraic topology, the Almgren--Dold--Thom Theorem, that the subspace $\mathcal{W}$ is not contractible inside $\cocyc{d}$.

\item If we choose the triangulation $\mathcal{T}$ sufficiently finely,
 then for every point $a\in \R^d$, there is a vertex $x$ of $\mathcal{T}$ with $F_a = F_x$. Thus, if no point of $\R^d$ is covered by ``many'' $d$-simplices of $P$, then all sets $F_x$ are ``small.'' If this is so,
 then by purely combinatorial means, we can define a concrete way of contracting the subspace $\mathcal{W}$ to a single point inside $\cocyc{d}$; see Figure~\ref{fig:IllustrationContraction}. This is a contradiction. 
\begin{figure}
\begin{center}
\includegraphics{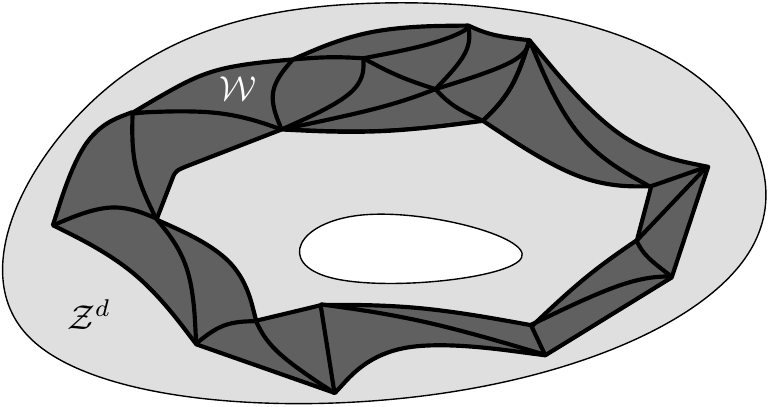}\quad \includegraphics{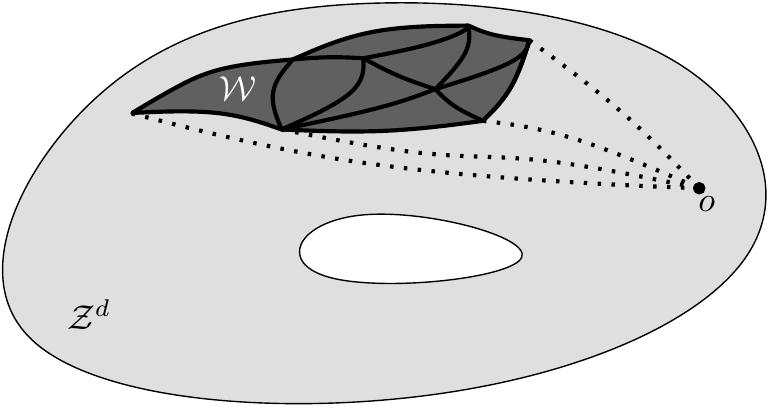}
\caption{A schematic illustration of the last two steps of the argument: $\mathcal{W}$ is not contractible inside $\cocyc{d}$, but if no point in $\R^d$ were covered by sufficiently many $d$-simplices of $P$ then we could contract $\mathcal{W}$ inside $\cocyc{d}$ to a single point.\label{fig:IllustrationContraction}}
\end{center}
\end{figure}
Thus, some point must be covered by many $d$-simplices.
\end{enumerate}

We now proceed to discuss the above steps in more detail.

\subsection{Simplicial Sets and the Space of Cocycles}
\label{subsec:simplicialsets}
\emph{Simplicial sets}\footnote{These objects also have many other names commonly found in the literature, including \emph{complete semisimplicial complexes} \cite{EilenbergZilber:SemisimplicialComplexes-1950}. Gromov uses the terminology \emph{semisimplicial spaces}.
}
are a generalization of simplicial complexes. As in the case of a simplicial complex, a simplicial set is built from $0$-simplices (vertices), $1$-simplices (edges), $2$-simplices (triangles), and higher-dimensional simplices.
One starts with the vertices, then glues each edge to one or two vertices by its endpoints, then one attaches triangles to vertices or edges along their boundaries, etc. In contrast to simplicial complexes, the attaching may involve various identifications. For instance, both endpoints of an edge may be attached to the same vertex, and two or more $i$-simplices in a simplicial set may have the same boundary. In this respect, simplicial sets can be thought of, in a first approximation, as higher-dimensional analogues of multigraphs with loops. On the other hand, in contrast to general cell complexes, there are restrictions as to what kind of attaching maps are allowed%
\footnote{Roughly speaking, one can think of each of the original simplices as having an ordered set of vertices. The attaching maps are linear maps induced by weakly monotone (order-preserving) maps between the vertex sets of the simplices.}, which makes simplicial sets more combinatorial than general cell complexes. We refer the reader to the article by Friedman \cite{Friedman:IntroductionSimplicialSets-2008} 
for a very clear and accessible introduction to simplicial sets and to the book by May \cite{May:SimplicialObjects-1992} for a detailed treatment (further references can be found in Friedman's article). 

The key object in Gromov's method is the \emph{space of $d$-dimensional cocycles}, which we denote\footnote{
For those readers who wish to read \cite{Gromov:SingularitiesExpandersTopologyOfMaps2-2010} in conjunction with the present one, we remark that Gromov uses the notation $\textrm{cl}^d$ or $\textrm{cl}^d_{\textup{sms}}$ for the space we denote by $\cocyc{d}$.} by $\cocyc{d}=\cocyc{d}(\Delta^{n-1})$ and which is a simplicial set defined as follows.

The vertices of $\cocyc{d}$ are the $d$-dimensional cocycles (of the simplex $\Delta^{n-1}$), i.e., subsets $F\subseteq \binom{V}{d+1}$ such that $\delta F=0$.

The edges of $\cocyc{d}$ are given by two (not necessarily distinct) $d$-cocycles $F_1$ and $F_2$ and a set $F_{12}\subseteq \binom{V}{d}$ such that $\delta F_{12}=F_1+F_2$. We stress that $F_{12}$ and at least one of the $F_i$ are necessary in order to uniquely define an edge in $\cocyc{d}$. If there is a $F_{12}'$ with $\delta F_{12}'=F_1+F_2$ then it defines a different edge  of $\cocyc{d}$ connecting the same pair of vertices. On the other hand, if $F_1'$ and $F_2'$ are another pair of $d$-cocycles with $\delta F_{12} =F_1'+F_2'$, then the same $F_{12}$ yields a different edge of $\cocyc{d}$ connecting a different pair of vertices.  

In the next step, a triangle of $\cocyc{d}$ is given by a triple of $d$-cocycles $F_1$, $F_2$, $F_3$, a triple of sets $F_{ij}\subseteq \binom{V}{d}$ and a set $F_{123} \subseteq \binom{V}{d-1}$ such that
\begin{enumerate}
\item[(i)] $\delta F_{ij}= F_i + F_j$, $1\leq i < j \leq 3$, and 
\item[(ii)] $\delta F_{123} = F_{12}+ F_{13} + F_{23}$. 
\end{enumerate}
The $F_i$ and the $F_{ij}$ define three (not necessarily distinct) vertices and three (not necessarily distinct) edges of $\cocyc{d}$ that form the boundary of a triangle, and together with this other data, $F_{123}$ defines a triangle glued in along that boundary. Again, there may be other $F_{123}'$ with the same coboundary, which define different triangles glued to the same boundary (a higher-dimensional analogue of a multiedge), and there may be a different set of $F_i'$ and/or $F_{ij}'$ which also satisfy conditions (i) and (ii); if so, they yield, together with $F_{123}$, a different triangle of $\cocyc{d}$, glued to a different boundary.

One can continue this definition inductively\footnote{In the beginning of \cite[Section 2.2]{Gromov:SingularitiesExpandersTopologyOfMaps2-2010} Gromov also gives an equivalent definition along the lines of the usual formal viewpoint of simplicial sets as functors from the category of finite totally ordered sets and monotone maps to the category of sets, as in \cite{May:SimplicialObjects-1992}.} for simplices of arbitrary dimension $r$. The case of $(d+1)$-simplices in $\cocyc{d}$, deserves special attention, however (due to the exceptional behavior of the coboundary operator in dimension zero, i.e., the fact that $V$ is not considered a coboundary, which was mentioned in Section~\ref{sec:basics}). A $(d+1)$-simplex in $\cocyc{d}$ is given by the following data (see Figure~\ref{fig:3SimplexZ2}): for each $i=0,1,\ldots,d$
 and each $A\in \binom{[d+2]}{i+1}$, there is 
a set $F_A \in \binom{V}{d+1-i}$ 
(i.e., a set of $(d-i)$-dimensional faces) such that 
$$\delta F_A =\sum_{B \in \partial A} F_B \qquad \textrm{ for }  0 \leq i \leq d,$$
and 
$$\sum_{A \in \binom{[d+2]}{d+1}} F_A=0.$$
\begin{figure}[tb]
\begin{center}
\includegraphics{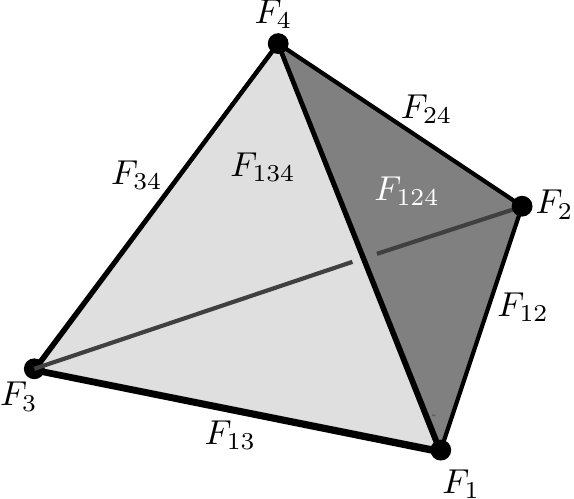}
\caption{An illustration of a $3$-dimensional simplex in $\cocyc{2}$. It is given by four $2$-dimensional cocycles $F_i \in \binom{V}{3}$, six sets $F_{ij} \in \binom{V}{2}$, and four sets$F_{ijk} \in \binom{V}{1}$ (with pairwise distinct indices $i,j,k$  running between $1$ and $4$) that satisfy the relations $\delta F_{ij}= F_i + F_j$,  $\delta F_{ijk}=F_{ij}+F_{jk}+F_{ik}$, and $\sum_{ijk} F_{ijk}=0$. 
\label{fig:3SimplexZ2}}
\end{center}
\end{figure}

\subsection{Intersections and Cocycles}
Let $P =\{p_1,p_2,\ldots, p_n\} \subseteq \R^d$ be a labeled set of $n$ points in general position. We think of $V=[n]$ as the set of ``labels'' of $P$. The goal of this section is to define the subspace $\mathcal{W}=\mathcal{W}(P) \subseteq \cocyc{d}$ for this set $P$.

Let $A=a_0a_1\ldots a_k$ be a $k$-dimensional simplex in $\R^d$ that is in general position w.r.t. $P$, i.e., no $i$-face of $A$ intersects any $(d-i-1)$-simplex spanned by $P$, $0\leq i \leq k$. We define
$$F_A:= \{\{i_0, i_{1}\ldots, i_{d-k} \} \in {\textstyle \binom{V}{d+1-k}}: A \cap p_{i_0}p_{i_1}\ldots p_{i_{d-k}} \neq \emptyset\}.
$$
That is, we consider the $(d-k)$-simplices spanned by $P$ that are intersected by $A$. Each such simplex is of the form $p_{i_0}p_{i_1}\ldots p_{i_{d-k}}$ for some $(d+1+k)$-tuple $\{i_0, i_1 \ldots, i_{d-k} \} \in {\textstyle \binom{V}{d+1-k}}$ of labels, and $F_A$ consists precisely of these tuples. For simplicity, we will also say that $F_A$ corresponds to the set of $(d-k)$-simplices of $P$ intersected by $A$.

Thus, for a point $x \in \R^d$ in general position w.r.t.\ $P$, $F_x =\{ \{i_0,i_1,\ldots, i_d\} \in \binom{V}{d+1}: x \in p_{i_0}p_{i_1}\ldots p_{i_d}\}$ corresponds to the set of $d$-simplices of $P$ that contain $x$. If $xy$ is a segment
in general position, then $F_{xy}$ corresponds to the set of $(d-1)$-simplices of $P$ that intersect $xy$, etc. 

As remarked above, the sets $F_x$ are always cocycles, i.e., $\delta F_x=0$, and the sets $F_{xy}$ satisfy
$F_x + F_y =\delta F_{xy}$. More generally, we have:




\begin{lemma}
\label{lem:intersection-duality}
Let $A=a_0a_1\ldots a_k$ be a $k$-simplex in $\R^d$ that is in general position w.r.t.\ $P$. 
Then
$$\delta F_A=F_{\partial A}:=F_{A_0} + F_{A_1} + \ldots + F_{A_k},$$ 
where $A_i=a_0\ldots a_{i-1} a_{i+1}\ldots a_k$ is the $(k-1)$-face of $A$ obtained by dropping vertex $a_i$.
\end{lemma}
\begin{proof}
Consider a $(d-k+2)$-tuple in $f \subseteq V$ corresponding to a $(d-k+1)$-dimensional simplex $\sigma$ spanned by $P$.
By general position, this $(d-k+1)$-dimensional simplex $\sigma$ is either disjoint from $A$, or it intersects $A$ in a line segment. Each endpoint of this line segment $\sigma \cap A$ is of one of two types: either such an endpoint arises as the intersection of $\sigma \cap A_i$ of $\sigma$ with some facet of $A$, or as the intersection $\sigma_j \cap A$ of $A$ with some facet of $\sigma$. If the intersection $\sigma \cap A$ is empty or if both endpoints are of the same kind then $f$ does not contribute to either side of the claimed identity. If there is one endpoint of each type, then $f$ contributes to both sides of the identity.
\end{proof}

If we apply the preceding lemma to a $(k-1)$-face $A_i$ of $A$, we see that
$$\delta F_{A_i}= \sum_{j\neq i} F_{A_{ij}},$$ where $A_{ij}$ is the $(k-2)$-face of $A$ obtained by dropping vertices $y_i$ and $y_j$.
Iterating this, we see that $A$, together with all its faces, defines a $k$-dimensional simplex in 
$\cocyc{d}$, which we denote by $\Delta_P(A)$. Note that, in particular, the vertices of $\Delta_P(A)$ correspond to the $d$-cocycles $F_{a_i}$ associated with the vertices $a_i$ of $A$.

\smallskip

Now we proceed to define the space $\mathcal{W}$. We fix a $d$-dimensional bounding simplex $B$ that contains $P$ in its interior. We also choose and fix a triangulation of $B$ that is in general position with respect to $P$ and that is 
\emph{sufficiently fine}, in the sense that
\begin{enumerate}
\item[(a)] for every point $q\in \R^d$ there is a vertex $x$ of $\mathcal{T}$ with $F_q=F_x$, and 
\item[(b)] any simplex $A$ in the triangulation of dimension $\dim A=k>0$ intersects $o(n^{d-k+1})$ of the $(d-k)$-simplices of $P$, i.e., $\|F_A\|=o(1)$.
\end{enumerate}
Now we complete this triangulation of $B$ to a triangulation $\mathcal{T}$ of the $d$-dimensional sphere\footnote{The somewhat ad-hoc device of introducing a bounding simplex and passing to a triangulation of the sphere can be avoided by using so-called homology with infinite supports (as Gromov does), but we opted for the ad-hoc method to keep the treatment more elementary.} $\s^d$ by adding a point at infinity and coning from this point over the boundary of $B$; see Figure~\ref{fig:triangulation}.
\begin{figure}[tb]
\begin{center}
\includegraphics{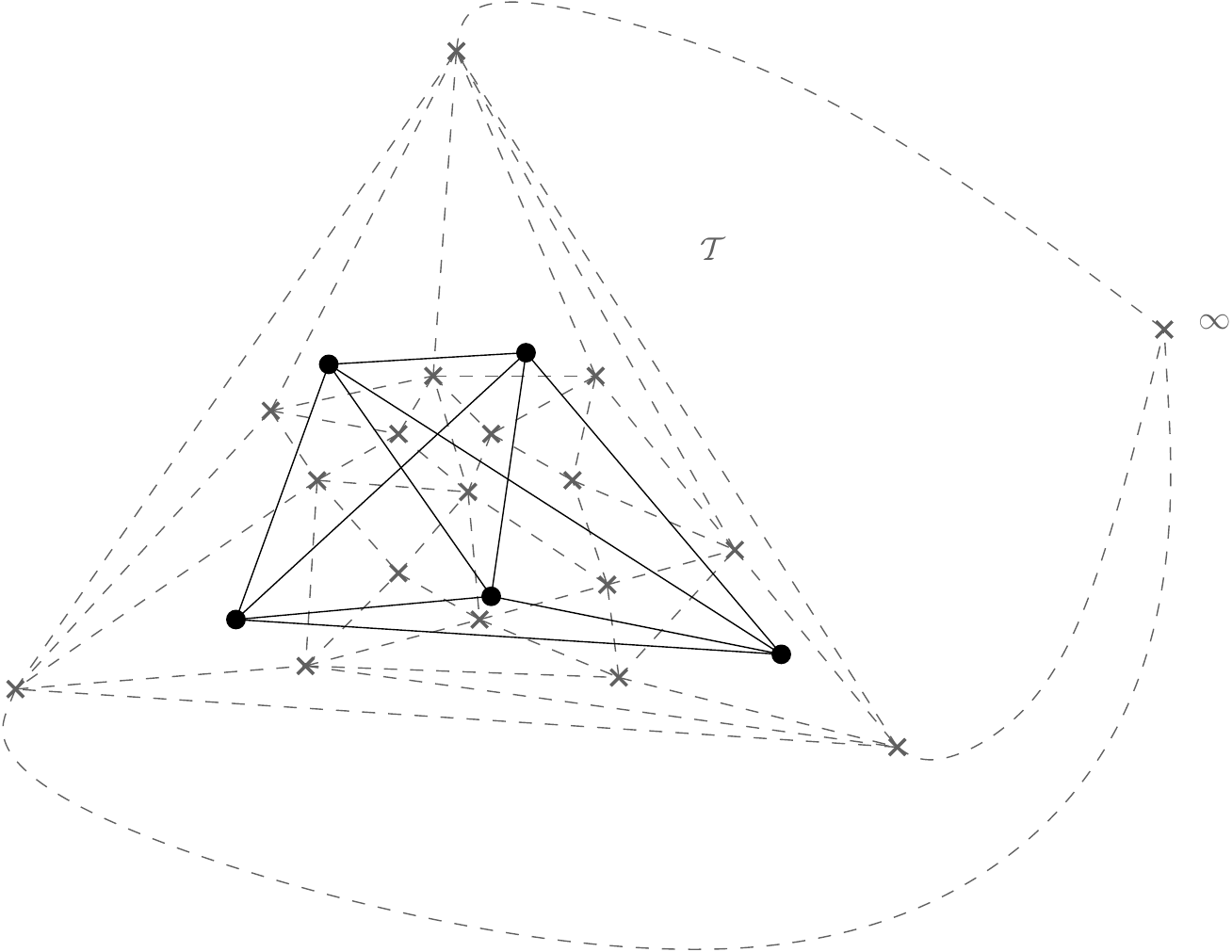}
\caption{\label{fig:triangulation} A set of five points in the plane and the line segments spanned by them (depicted by bold dots and solid segments), and a triangulation $\mathcal{T}$ of a bounding simplex $B$ plus a vertex at infinity (depicted by crosses and dashed lines).}
\end{center}
\end{figure}

We define a subspace $\mathcal{W}$ of $\cocyc{d}$ by taking the formal sum of the $d$-simplices $\Delta_P(A)$ over all $d$-simplices $A$ of $\mathcal{T}$ (note that for all simplices $A$ involving the vertex at infinity, we have $F_A=0$). In other words, a $d$-simplex of $\cocyc{d}$ is included in $\mathcal{W}$ if it is equal to $\Delta_P(A)$ for an odd number of $d$-simplices $A$ of $\mathcal{T}$. (Formally, in homological terminology, $\mathcal{W}$ is a $d$-dimensional simplicial cycle in $\cocyc{d}$.) We stress that $\mathcal{W}$ is determined by the $d$-simplices of $\mathcal{T}$, not by the vertices of $\mathcal{T}$. 

As we have described it, the subspace $\mathcal{W}$ of $\cocyc{d}$ depends not only on $P$, but also on the triangulation $\mathcal{T}$ that we have chosen. It turns out that for any two choices of triangulations, the corresponding subspaces $\mathcal{W}$ are equivalent in a suitable sense (any two such cycles are homologous); see Section~\ref{subsec:abstract}.

\subsection{Nontriviality} 
The key fact upon which Gromov's method hinges is that the subspace $\mathcal{W}$ defined in the previous subsection is always nontrivial, in the following sense:

\begin{quotation}
\textbf{Key Fact.} The subspace $\mathcal{W}$ defined above cannot be contracted inside $\cocyc{d}$.
\end{quotation}
(More formally, one has the stronger statement that $\mathcal{W}$ is homologically nontrivial, i.e., that it is not a homological boundary inside $\cocyc{d}$.)

Gromov deduces this fact from  what he calls \emph{the algebraic version of the Almgren--Dold--Thom theorem} (see \cite[Section~2.2]{Gromov:SingularitiesExpandersTopologyOfMaps2-2010}), but 
we have not been able to locate the ADT theorem in a suitable form and with proof in the literature.\footnote{The paper by Almgren \cite{Almgren:HomotopyGroupsIntegralCycleGroups-1962} that Gromov cites works in the setting of geometric measure theory and is about integral currents and cycles; an older paper by Dold and Thom \cite{DoldThom-Quasifaserungen-1958} works in the setting of simplicial sets, but only establishes a special case. It may well be that the theorem is well-known and clear to experts in the field, but not to us.} So
 we will treat the Key Fact as a black box in our presentation.

\subsection{Coning in the Space of Cocycles and the Proof of Proposition~\ref{prop:GromovFillingBaranyConstant}}\label{s:coning}

To prove Proposition~\ref{prop:GromovFillingBaranyConstant}, one argues that if $\| F_y\|$ were ``too small'' for all vertices $y$ of the triangulation $\mathcal{T}$, then the space $\mathcal{W}$ could be contracted to a point inside $\mathcal{W}$---a contradiction. 

In order to show contractibility, we have the following simple \emph{coning argument}: Suppose there is a vertex $o$ in $\cocyc{d}$ such that we can inductively construct a \emph{cone} $o\ast \mathcal{W}$ in $\mathcal{Z}^d$. 
That is, suppose we do the following, by induction on the dimension $i$: For each $i$-simplex $\tau$ that is a face of at least one $d$-face in $\mathcal{W}$, select an $(i+1)$-simplex $o\ast \tau$ in $\cocyc{d}$ that has $\tau$ as an $i$-face and $o$ as the remaining vertex, in such a way that
\begin{equation}\label{eq:coning}
\partial (o \ast \tau)=\tau + o \ast \partial \tau.
\end{equation}
This condition means that our choices for higher-dimensional faces have to be consistent with what we have already committed to for lower-dimensional faces. We note that if $\cocyc{d}$ were just a simplicial complex, then for each $\tau$ there would be either a unique choice for $o\ast \tau$, or none at all, but for simplicial sets, there may be many choices. We also remark that $o\ast \tau$ may be a degenerate simplex, in the sense that $o$ already appears as a vertex of $\tau$.


Given a coning, one can contract $\mathcal{W}$ to $o$, by continuously moving each simplex $\tau$ of $\mathcal{W}$ towards $o$ inside $o\ast \tau$.

To perform the coning for $\mathcal{W}$, we need to fix a vertex $o$ in $\cocyc{d}$ and, for every  $k$-simplex $A$ of the triangulation $\mathcal{T}$, 
choose a $(k+1)$-simplex $o\ast \Delta_P(A)$ in such a way that the coning condition (\ref{eq:coning}) is satisfied. That is, inductively, for every $k$-simplex $A$ of $\mathcal{T}$, we have to choose a $(d-k-1)$-cochain $F_{Ao} \in \binom{V}{d-k}$ such that 
$$\delta F_{Ao} = F_A + \sum_{B} F_{Bo},$$
where the sum is over all $(k-1)$-faces $B$ of $A$.

Using the {\cofilling} profile of $\Delta^{n-1}$, we show that we can do this if all $\|F_y\|$ are small, thus obtaining a contradiction. 

We choose $o$ to be the vertex of $\cocyc{d}$ corresponding to the zero $d$-cocycle $0$ in $\Delta^{n-1}$. For every vertex $y$ of the triangulation, 
let $F_y$  be the corresponding vertex in $\cocyc{d}$ 
($d$-cocycle in $\Delta^{n-1}$). We pick an arbitrary  
\emph{minimal} $(d-1)$-cochain $F_{yo}$ with $\delta F_{yo} = F_y$ $(=F_y +0)$.
 By minimality, we have $\|F_y\| \geq \varphi_d (\|F_{oy}\|)$.

Next, consider an edge $xy$ in the triangulation $\mathcal{T}$. The corresponding $(d-1)$-cochain $ F_{xy}$ satisfies 
$\delta F_{xy}= F_x + F_y$ (Lemma~\ref{lem:intersection-duality}).
 It follows that  
$$\delta (F_{xy} + F_{xo} + F_{yo})= F_x+F_y+F_x+F_y=0.$$
Now we pick a minimal $(d-2)$-cochain $F_{xyo}$ 
such that $\delta F_{xyo} =  F_{xy} + F_{xo} + F_{yo}$. 
It follows that $\|F_{xy} + F_{xo} + F_{yo}\| \geq \varphi_{d-1}(\|F_{xyo}\|)$. Moreover, by our choice of the triangulation, we
have $\|F_{xy}\|=o(1)$. Thus, up to an $o(1)$ additive error, $\|F_{xo} + F_{yo}\| \geq \varphi_{d-1}(\|F_{xyo}\|)$, hence
$$\max\{\|F_{xo}\|,\| F_{yo}\|\} \geq \tfrac{1}{2} \varphi_{d-1}(\|F_{xyo}\|),$$
and so
$$\max\{ \|F_x\|, \|F_y\|\} \geq \varphi_d(\tfrac{1}{2} \varphi_{d-1}(\|F_{xyo}\|)$$
(where we suppress $o(1)$ additive error terms in both formulas). 
\begin{figure}[tb]
\begin{center}
\includegraphics{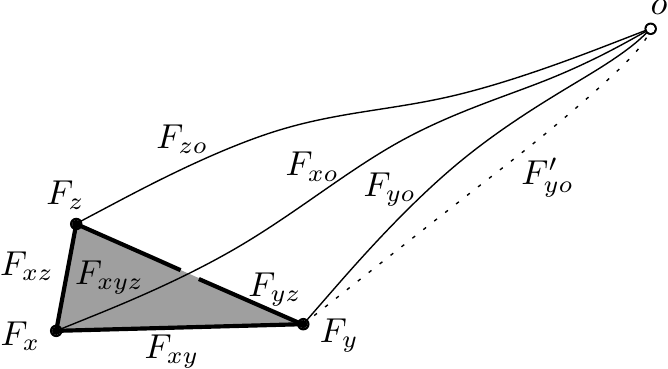}
\caption{\label{fig:coning} An illustration of the coning for the boundary of a triangle $xyz$ of $\mathcal{T}$. We use a simplified labeling of the resulting simplices in $\cocyc{d}$. For example, the $0$-chain $F_{xzo}$ by itself does not determine a triangle of $\cocyc{d}$, but only does so together with the $2$-cocycles $o, F_x, F_z$ and the $1$-chains $F_{xz}, F_{xo}, F_{zo}$ that are already determined. We also stress that the coning involves making choices, e.g. between $1$-cochains $F_{yo}$ and $F_{yo}'$ that determine different edges connecting $F_y$ and $o$, and once we have made a choice, we must stick to it when choosing higher-dimensional faces.}
\end{center}
\end{figure}

In the next step, we consider a triangle $xyz$ of $\mathcal{T}$. 
By our assumption on $\mathcal{T}$, we have $\|F_{xyz}\|=o(1)$.
By the choice of $F_{xyo}$, $F_{yzo}$, and $F_{zxo}$,
and using Lemma~\ref{lem:intersection-duality}, we obtain
$$\delta (F_{xyz} + F_{xyo} + F_{yzo}+F_{zxo})=0,$$
and so we can choose a minimal $(d-3)$-chain $F_{xyzo}$ with 
$$\delta F_{xyzo} =  F_{xyz} + F_{xyo} + F_{yzo}+F_{zxo}.$$
See Figure~\ref{fig:coning} for an illustration. Reasoning as before, we obtain $\max\{ \|F_{xyo}\|,  \| F_{yzo}\| , \|F_{zxo}\|\} \geq \tfrac{1}{3} \varphi_{d-2}(\|F_{xyzo}\|)$, and hence
$$\max\{ \|F_x\|, \|F_y\|, \|F_z\|\} \geq \varphi_d(\tfrac{1}{2} \varphi_{d-1}(\tfrac{1}{3} \varphi_{d-2}(\|F_{xyzo}\|))).$$

We can continue this argument by induction. In the final step, consider a $d$-face $A$ of the triangulation.
Corresponding to it, there is a $0$-cochain $F^d_A$. Moreover, for every $(d-1)$-face $B$ of $A$, we have already constructed a $0$-cochain $F_{Bo}$ such that 
$$ \max_{x\in B} \|F_x\| \geq \varphi_{d}(\tfrac12 \varphi_{d-1}(\tfrac13 \varphi_{d-2}(\ldots \tfrac 1d \varphi_1(\|F_{Bo}\|)\ldots)))$$
and $\delta (F_a + \sum_B F_{Bo})=0$. Thus, $F_a + \sum_B F_{Bo}$ is a $0$-cocycle, and so
 it must be either $0$ or all of $V$, the whole vertex set of $\Delta^{n-1}$.
 In the former case, we can complete the coning for $A$ by setting 
$F_{Ao}=0$. If we could do this for all $d$-faces $A$ of $\mathcal{T}$, we would
be able to complete the coning, thus reaching a contradiction.

Therefore, there must be a $d$-face $A$ such that $F_A + \sum_B F_{Bo}=V$. Since $\|F_A\|=o(1)$, it follows that for some $B\subset A$, 
we must have $\|F_{Bo}\| \geq \tfrac{1}{d+1}$. Maximizing over all vertices of $\mathcal{T}$, we conclude
$$ \max_{x} \|F_x\| \geq \varphi_{d}(\tfrac12 \varphi_{d-1}(\tfrac13 \varphi_{d-2}(\ldots \tfrac 1d \varphi_1(\tfrac{1}{d+1})\ldots))).$$

This completes (our outline of) the proof of Proposition~\ref{prop:GromovFillingBaranyConstant} in the affine case.

\subsection{Gromov's Method in the Topological Setting}
\label{subsec:abstract}

As was mentioned in the introduction, the setting of an $n$-point set $P \subseteq \R^d$ considered in the previous subsections corresponds to an affine map from the $(n-1)$-simplex $\Delta^n$ into $\R^d$.

Gromov's method applies to much more general situations. More precisely, it allows for the setting to generalized in several ways,
as we now sketch. 
\begin{enumerate}
\item The simplex $\Delta^{n-1}$ can be replaced by an arbitrary (finite) simplicial complex $X$.

What is needed are lower bounds on the {\cofilling} profile of $X$, which is defined as follows:
For $k\geq 0$, let $X_k$ be the set of $k$-dimensional faces of $X$. For each $k$, we have
the coboundary operator of $X$ which maps a subset $E\subseteq X_k$ to 
$$\delta E:=\{f\in X_{k+1}: f \textrm{ contains an odd number of } e\in E\}.$$
We can identify $E \subseteq X_k$ with a $0/1$-vector indexed by $X_k$, i.e., with an element of the vector space $\Z_2^{X_k}$ over the $2$-element field $\Z_2$. In more usual (co)homological terminology,
this latter vector space is denoted by $C^{k}(X;\Z_2)$ and called the space of \emph{$k$-dimensional cochains} (with $\Z_2$-coefficients), and $\delta$ is a linear map $C^k(X;\Z_2)\rightarrow C^{k+1}(X;\Z_2)$.

For a $k$-dimensional cochain $E$, we define $\|E\|:=|E|/|X_k|$ as the normalized support size of $E$ as before,
and 
$$
\varphi^X_d(\alpha):=\min \{\|\delta E\|: E\in {\textstyle C^{\ast}(X;\Z_2)} 
\mbox{ minimal}, \|E\|\ge \alpha\},
$$
where $E$ is minimal if $\|E\| \leq \|E+ \delta D\|$ for all $D\in C^\ast(X;\Z_2)$.

The space $\cocyc{d}(X)$ of $d$-dimensional cocycles of $X$ is defined completely analogously to the case of the simplex.

\item The target space $\R^d$ can be replaced by an arbitrary triangulated $d$-dimensional manifold $Y$, or, even more generally, a $\Z_2$-homology manifold. What we need is to be able to compute \emph{intersection numbers} (modulo 2) between $k$-dimensional chains and $(d-k)$-dimensional chains. Equivalently, we need that Poincar\'e duality (with $\Z_2$ coefficients) holds in~$Y$.

\item Instead of affine maps, we can allow for
arbitrary continuous maps $T\colon X\rightarrow Y$. Without loss of generality, one can think of $T$ as a piecewise linear map in general position. This is not necessary for the argument, but may help the reader's intuition. For instance, for such a map, the $T$-image of any $(d-k)$-simplex $\sigma$ of $X$ intersects any $k$-simplex $A$ of $Y$ in a finite number of points in the relative interior of $A$, and the (algebraic, $\Z_2$) intersection number of $T(\sigma)$ and $A$ is defined as the number of intersection points modulo~$2$.

\item Finally, instead of cohomology with $\Z_2$-coefficients, one can work with
other coefficient rings in the argument. Potentially, this might lead
to stronger bounds. We will not discuss this generalization, since
on the one hand, it is straightforward, and on the other hand, we would
have to talk about orientations everywhere.
\end{enumerate}

The basic structure of the proof remains the same, but we have to change the definition of the cochains $F_A$ appropriately, as follows: If $A$ is a $k$-dimensional simplex in general position, one defines  $F_A$ as the set of $(d-k)$-simplices of $X$ whose images under $T$ have odd intersection number with $A$.

Another way of interpreting this construction is as follows: Every $k$-simplex $A$ defines, via Poincar\'e duality, a $(d-k)$-cochain in $Y$, and this $(d-k)$-cochain pulls back under $T$ to a cochain $F_A$ in $X$. 

The basic identity
\begin{equation}
\label{eq:boundary-coboundary-intersection}
\delta F_A=F_{\partial A}:=F_{A_0} + F_{A_1} + \ldots + F_{A_k}
\end{equation}
still holds. (This just says that on the level of chains and cochains, Poincar\'e duality exchanges boundary and coboundary operators).

In other words, every $k$-simplex $A$ in general position defines a $k$-simplex $\Delta_T(A)$ in $\cocyc{d}(X)$ via the intersection number construction. If $Y$ is compact, i.e., has a finite triangulation, then the subspace $\mathcal{W}=\mathcal{W}(T) \subseteq \cocyc{d}(X)$ is defined as the formal $\Z_2$-linear combination of the $d$-simplices $\Delta_T(A)$, where $A$ ranges over all $d$-simplices in the triangulation of $Y$. In other words, a $d$-simplex of $\cocyc{d}(X)$ belongs to $\mathcal{W}$ if it equals $\Delta_T(A)$ for an odd number of $d$-simplices $A$ of the triangulation of $Y$.

An equivalent way of defining $\mathcal{W}$ this is as follows: The basic identity (\ref{eq:boundary-coboundary-intersection}) implies that the map $A\mapsto \Delta_T(A)$ commutes with the boundary operator, i.e., it is a chain map (with $\Z_2$-coefficients) from $Y$ to $\cocyc{d}(X)$. Thus, this map also induces a map in homology.

Let $[Y]$ denote the fundamental $d$-dimensional homology class (over $\Z_2$) of $Y$. If we fix a triangulation of $Y$, we can take a representative $d$-cycle for $[Y]$ that is the formal $\Z_2$-linear combination of all the $d$-simplices of the triangulation. (If $Y$ is not compact, as in the case $Y=\R^d$, then we have to work with homology with infinite supports.)  Then $\mathcal{W}=\mathcal{W}(f)$ is defined as the image under the map $\Delta_T$ of $[Y]$, i.e., formally, it is a $d$-dimensional cycle (with $\Z_2$-coefficients) in $\cocyc{d}(X)$.

As before, the Almgren--Dold-Thom Theorem implies that this cycle is homologically nontrivial, so it cannot be contracted. 

On the other hand, if every point of $Y$ were covered by the $T$-images of ``too few'' $d$-simplices of $X$, then the same combinatorial coning as before would yield a contradiction (with the precise meaning of ``too few'' depending on the {\cofilling} profile of $X$). Again, the combinatorial coning can also be viewed as an actual topological contraction of $\mathcal{W}$, considered as a subspace of $\cocyc{d}(X)$, to the point $o$. 

If $Y$ is unbounded, or if we are guaranteed that there is a point of $Y$ that is not covered by the image of $T$, then the same argument as before shows that
$$ \max_{y} \|F_y\| \geq \varphi^X_{d}(\tfrac12 \varphi^X_{d-1}(\tfrac13 \varphi^X_{d-2}(\ldots \tfrac 1d \varphi^X_1(\tfrac{1}{d+1})\ldots))),$$
where the maximum is over all vertices $y$ of the triangulation of $Y$.
(If $Y$ is not unbounded and if every point in $Y$ is covered by the $T$-images of some $d$-simplices of $X$, then we have to choose the apex $o$ of the coning differently (not as the empty cocycle), which yields a weaker bound.)

In particular, if $X$ is a \emph{$\Z_2$-cohomological expander}, i.e., if $\varphi_i^X(\alpha)/\alpha$ is bounded away from zero for all $i$, then there is some point of $Y$ that is covered by a positive fraction of all $d$-simplices of $X$ (with a constant depending on the cofilling profile of $X$).

We remark that for the coning argument, we again choose the triangulation $\mathcal{T}$ to be sufficiently fine with respect to the map $T$, i.e., we assume that for every simplex $A$ of the triangulation of dimension $\dim A >0$, we have $\|F_A\|=o(1)$. If the map $T$ is very complicated (i.e., if we need a very fine subdivision of $X$ to approximate $T$ by a PL map), then the triangulation $\mathcal{T}$ may require a huge number of simplices, so it is important that the whole argument is completely independent of the number of simplices in the triangulation $\mathcal{T}$.

\section{The {\cofilling} profile in the \boldmath$d=2$ case}
\label{sec:cofilling-2}

Here we prove Theorem~\ref{thm:cofilling-2}, which asserts that
$\varphi_2(\alpha) \ge f(\alpha):=\tfrac34\left(1-\sqrt{1-4\alpha}\right)(1-4\alpha)$.

Since $d=2$, we deal with the size of the coboundary
for an edge set $E$ of a graph. The minimality of $E$ means that
\emph{no edge cut has density more than~$\frac12$} in this graph;
in other words, for every $S\subseteq V$, the number of edges
of $E$ going between $S$ and $V\setminus S$ is at most 
$\frac 12|S|\cdot|V\setminus S|$.

As we have remarked at the end of Section~\ref{s:seidel},
the minimality of $E$ is a complicated property, computationally
hard to test, for example. So we will only use it for
singleton sets $S$.
Thus, we will actually show that $\|\delta E\|\ge
f(\alpha)$ (ignoring terms tending to $0$ as $n\to\infty$)
for every $E$ with $\|E\|\ge\alpha$ and
$\deg_E(v)\le \frac n2$ for all $v\in V$ (where $\deg_E(v)$ denotes
the number of neighbors of $v$ in the graph $(V,E)$).

Before proceeding with the proof of Theorem~\ref{thm:cofilling-2},
let us remark that this relaxation (i.e., ignoring
all non-singleton $S$) already prevents us from obtaining
a tight bound for $\varphi_2$. For example, let us partition $V$ into
sets $V_1,V_2,V_3$, where $|V_1|=\alpha n$ and $|V_2|=\frac n2$,
and let $E$ consist of all edges connecting
$V_1$ to $V_2$.
This $E$ is not minimal, but it does satisfy the degree condition.
One easily checks that $\|E\|=\alpha$ and $\|\delta E\|=3 \alpha(\frac12-\alpha)=\frac 32\alpha-3\alpha^2$, which is smaller than the suspected
tight bound for $\varphi_2$ from Proposition~\ref{p:ubbb}.
However, at least the leading term is correct.

Now we proceed with the proof. Given $E$, let $m_i$
denote the number of triples $f\in {V\choose 3}$ that contain
$i$ edges of $E$, $i=1,2,3$; we have $|\delta E|= m_1+m_3$ by
definition. An easy inclusion-exclusion consideration
shows that
\begin{equation}\label{e:pieform}
|\delta E| = (n-2)|E|-\sum_{v\in V}\deg_E(v)(\deg_E(v)-1)+ 4t,
\end{equation}
where $t$ denotes the number of triangles in the graph $(V,E)$.
Indeed, to check (\ref{e:pieform}), it suffices to discuss
how many times a triple $f\in{V\choose 3}$ containing 
exactly $i$ edges $e\in E$ contributes to the right-hand side,
$i=1,2,3$. For $i=1$, such an $f$ is only counted once in the term
$(n-2)|E|$, which counts the number of ordered pairs $(e,v)$,
where $e\in E$ and $v\in V\setminus e$.  A triple $f$ with $i=2$
is counted twice in the term $(n-2)|E|$, but it is also counted twice
in $\sum_{v\in V}\deg_E(v)(\deg_E(v)-1)$, and thus its total contribution is
zero. Finally, for $i=3$, where $f$ induces a triangle,
it is counted  three
times in $(n-2)|E|$, six times in $\sum_{v\in V}\deg_E(v)(\deg_E(v)-1)$,
and four times in $4t$, so altogether it contributes $+1$
as it should.

%

As the next simplification, we will ignore the triangles,
as well as the difference between $\deg_E(v)(\deg_E(v)-1)$
and $\deg_E(v)^2$, and we will
use (\ref{e:pieform}) in the form
\begin{equation}\label{e:truncpie}
|\delta E| \ge (n-2)|E|-\sum_{v\in V}\deg_E(v)^2.
\end{equation}
Since $|E|$ is given, it remains to maximize $\sum_{v\in V}\deg_E(v)^2$,
which is done in the next lemma.

\begin{lemma}\label{l:lobo2} 
Let $\alpha\le \frac 14$.
Let $(V,E)$ be a graph on $n$ vertices
with $|E|=\alpha{n\choose 2}$ and $\deg_E(v)\le \frac n2$ for all
$v\in V$. Then
$$
\sum_{v\in V}\deg_E(v)^2 \le \left(\tfrac\sigma4
(1+2\sigma-4\sigma^2)+o(1)\right)n^3,
$$
where $\sigma=(1-\sqrt{1-4\alpha})/2$.
\end{lemma}

\heading{Proof. } Let $(V,E)$ be the given graph with $n$ vertices and $|E|=\alpha\binom{n}{2}$ edges.
By a sequence of transformations that do not change the number of edges and that do not 
decrease  the sum of squared degrees, we convert it to a particular
form.

Let us number the vertices $v_1,\ldots,v_n$ so that
$d_1\ge d_2\ge \cdots\ge d_n$, where $d_i:=\deg_E(v_i)$.
We note that for $d_i\ge d_j$, we have $(d_i+1)^2+(d_j-1)^2 > d_i^2+d_j^2$, 
and thus a transformation that changes $d_i$ to $d_i+1$ and $d_j$ to $d_j-1$ and 
leaves the rest of the degrees unchanged increases the sum of squared degrees.

For an edge $e=\{v_i,v_j\}$ with $i<j$, we call $v_i$
the \emph{left end} of $e$ and $v_j$ the \emph{right end}. 

\begin{enumerate}
\item[(i)]
 {\em Let $k$ be such that $d_1=d_2=\cdots=d_k=\lfloor \frac n2\rfloor$,
while $d_{k+1}<\lfloor \frac n2\rfloor$. Then we may assume
that the left ends of all edges are among $v_1,\ldots,v_{k+1}$.}

Indeed, if there is an edge $\{v_i,v_j\}$, $i>k+1$,
we can replace it with the edge $\{v_{k+1},v_j\}$. 
This increases $\sum_{i=1}^{k+1} d_i$ (and possibly increases
$k$), so after finitely many
steps, we achieve the required condition.

\item[(ii)]
 {\em  We may assume that every two vertices among $v_1,\ldots,v_k$ are connected.}

Proof: We may assume that (i) holds.
Since we assume $\alpha\le \frac14$, we have $k\le \lfloor \frac n2\rfloor$.
Let us suppose $\{v_i,v_j\}\not\in E$, $1\le i<j\le k$.
Since $d_i=d_j=\lfloor \frac n2\rfloor$, each of $d_i,d_j$
is connected to at least two vertices among $v_{k+1},\ldots,v_n$.
So we may assume $\{v_i,v_{\ell}\}\in E$, $\{v_j,v_m\}\in E$,
$\ell,m\ge k+1$, $\ell\ne m$. We also have $\{v_{\ell},v_m\}\not\in E$
(according to (i)). Thus, we can delete the edges
$\{v_i,v_{\ell}\}$ and $\{v_j,v_{m}\}$ and add the edges
$\{v_i,v_j\}$ and $\{v_{\ell},v_m\}$.
\immfig{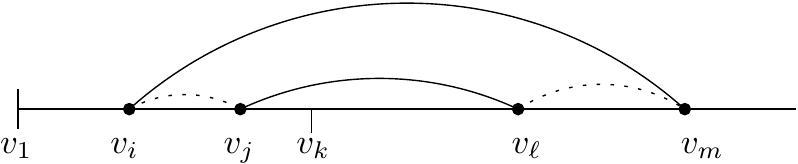}
This increases the number of edges on $\{v_1,\ldots,v_k\}$,
which cannot decrease by the transformations in (i), so after
finitely many steps, we achieve both (i) and (ii).
 
\item[(iii)] {\em We may assume that the right neighbors of each $v_i$, $1\le i\le k$, 
form a contiguous interval $v_{i+1},v_{i+2},\ldots,v_{\lfloor n/2\rfloor+1}$.}

Indeed, if $v_i$ is connected to some $v_{\ell+1}$ and not to $v_{\ell}$,
$\ell>k$, we can replace the edge $\{v_i,v_{\ell+1}\}$
with $\{v_i,v_{\ell}\}$. This increases the sum of squared vertex degrees,
and thus after finitely many steps, we can achieve (i)--(iii).
\end{enumerate}

A graph satisfying (i)--(iii) is almost completely determined
by its number of edges, except possibly for the neighbors
of the vertex $v_{k+1}$. Each of $v_1,\ldots,v_k$
is connected to the first $\frac n2$ vertices,
and there are no other edges, except possibly for those
incident to $v_{k+1}$.  Counting the left ends of edges,
we have $\alpha{n\choose2}=|E|=kn/2-{k\choose2}+O(n)$.

Writing $k=\sigma n$, we obtain  $\sigma=(1-\sqrt{1-4\alpha})/2+o(1)$.
The sum of the squared degrees is then
$k\frac{n^2}4+(\frac n2-k)k^2+O(n^2)=\left(
\frac\sigma 4(1+2\sigma-4\sigma^2)+o(1)\right)n^3$.
Lemma~\ref{l:lobo2} is proved.
\proofend

\medskip
Theorem now follows immediately from Lemma~\ref{l:lobo2} using (\ref{e:truncpie}). \proofend

\heading{A promising relaxation? }
As we have seen, in order to establish the tightness
of the upper bound from Proposition~\ref{p:ubbb},
one has to use the minimality condition in a stronger way
than we did in the above proof. On the other hand, it 
seems possible that the other relaxation we have made in that
proof, namely, ignoring triangles, need not cost us anything.
In other words, while 
in $\delta E$ we count triples containing 1 or 3 edges,
perhaps the example in Proposition~\ref{p:ubbb}
also minimizes, over all all minimal $E$ of a given size, 
the number of triples containing exactly one edge.
This might be easier to prove, and triangles would be dealt with
implicitly, since the example has no triangles.

\heading{On Gromov's ``$\frac23$-bound''. }
Sec.~3.7 of ~\cite{Gromov:SingularitiesExpandersTopologyOfMaps2-2010}
claims the lower bound
$\|(\partial^1)^{-1}_{\rm fil}\|(\beta)\le\frac2{3(1-\sqrt{\beta})}$
(which would yield
$
\varphi_2(\alpha) \ge \tfrac 32\alpha -(\tfrac 32)^{3/2}\alpha^{3/2}+\tfrac98\alpha^2-O(\alpha^{5/2})
$.

The argument as given doesn't seem to work, however 
(although it is also possible that we misunderstood something).
It is supposed to be based on an inequality 
(the fifth displayed formula
in Sec.~3.7, derived from the Loomis--Whitney inequality),
which seems correct and is re-stated for $i=1$ two lines below.
In the language of graphs, the $i=1$ case is equivalent to
$\sum_{\{u,v\}\in{V\choose 2}} \deg_E(u)\deg_E(v)\le 2\frac {n-1}n|E|^2$.

However, the proof of the ``$\frac23$-bound'' below
seems to employ a similar inequality for $i=2$, which would
claim that $\sum_{\{u,v\}\in{V\choose 2}}\sqrt{\deg_E(u)\deg_E(v)}\le
2|E|^{3/2}$. This is false, though (a graph as in the upper bound
example, i.e., a complete bipartite graph with parts of very
unequal size, is a counterexample). Probably this kind of proof
can be saved, since it seems sufficient to take the
last sum over $\{u,v\}\in E$, instead of all pairs, and then
such an inequality is apparently true (but doesn't seem
to follow from Loomis--Whitney in a direct way).

\section{The {\cofilling} profile  for \boldmath$d>2$}
\label{sec:cofilling-ge2}




In this section we prove Theorem~\ref{t:3ub}, a lower bound
on $\varphi_3$. We begin with an auxiliary fact concerning
links of vertices.

\begin{obs}\label{o:minlink}
 If $E\subseteq {V\choose d}$ is a minimal
system, then $\lk(v,E)$ is also minimal, for every vertex $v\in V$.
\end{obs}

\heading{Proof. } 
For a set system $F$ on $V$, let us write $F_{\setminus v}:=
\{s\in F:v\not\in F\}$.

We want to verify that for each $C\subseteq {V\choose d-2}$,
$\lk(v,E)$ contains at most half of $\delta C$. The sets of
$\lk(v,E)$ do not contain $v$, and so only sets of $(\delta C)_{\setminus v}$
may belong to $\lk(v,E)$. But we have $(\delta C)_{\setminus v} =
(\delta C_{\setminus v})_{\setminus v}$, and so we may restrict our attention
to systems $C$ whose sets all avoid~$v$.

Let $D:=C*v:=\{c\cup \{v\}:c\in C\}$. We have $\delta D=
((\delta C)_{\setminus v})*v$. Now $E$ contains at most half of the sets
of $\delta D$ by minimality, so $\lk(v,E)$ contains
at most half of the sets of $(\delta C)_{\setminus v}$ as claimed.
\proofend
\bigskip

In the proof of Theorem~\ref{t:3ub}, we consider a minimal
system $E\subseteq {V\choose 3}$, $\|E\|=\alpha$,
and we want to show that it has a large coboundary.
Conceptually, the proof splits into two cases:
The first one deals with the situation where most of the 
triples $e\in E$ are incident to vertices of
very large degrees. The second one concerns the situation
where the maximum vertex degree is not much larger than the
average vertex degree.

\heading{Dealing with high-degree vertices. }
We begin with the first case, with a significant share
of high-degree vertices.  Here we rely on the basic cofilling bound,
which we are going to apply to the links
of high-degree vertices (the links are minimal by 
Observation~\ref{o:minlink}). From the sets in the coboundaries of the links
we are going to obtain sets of $\delta E$; some care is needed
to avoid counting a single $f\in \delta E$ several times.

It is interesting to note that in this way we get a lower bound for
$\varphi_3$ that has a correct limit behavior as $\alpha\to 0$, 
 although for the links we employ
the basic cofilling bound, which is far from correct
for small $\alpha$'s. This can be explained as follows:
for those systems $E$ that are near-extremal for $\varphi_3$,
the relevant vertex links are so large that the basic cofilling bound
is almost   tight for them.

We present the first part of the proof 
for $d$-tuples instead of triples, since specializing
to triples would not make the argument any simpler.

\begin{lemma}\label{l:highdeg}
 Let $E$ be a minimal
system of $d$-tuples on $V=\{v_1,v_2,\ldots,v_n\}$,
let $\alpha:=\|E\|$,
let $r=\beta n$ be a parameter, and let 
$E_{\rm hi}\subseteq E$ consist of those $e\in E$ that
contain at least one vertex among $v_1,\ldots,v_r$.
Let $F_{\rm hi}$ be the set of those $f\in\delta E$ that contain 
at least one vertex among $v_1,\ldots,v_r$.
Then
$$
\|F_{\rm hi}\| \ge \frac {d+1}d \alpha_{\rm hi} -
 \frac{(d+1)d}2 \beta^2-(d+1)\alpha\beta-
O(n^{-1}), 
$$
where $\alpha_{\rm hi}:=\|E_{\rm hi}\|$.
\end{lemma}

\heading{Proof. } Let $v\in V$ be a vertex, and let
us write $L_v:=\lk(v,E)$.
Then, by Observation~\ref{o:minlink}, $L_v$ is minimal,
and thus $\|\delta L_v\|\ge \|L_v\|$ by the basic bound on 
$\varphi_{d-1}$. In terms of cardinalities, we
can write this inequality as
$|\delta L_v|\ge \frac{n-d+1}d|L_v|$.

Let us consider some $e\in \delta L_v$. We observe that
if $v\not\in e$ and $e\not\in E$, then $e\cup\{v\}\in\delta E$.
There are exactly $|L_v|$ sets $e\in \delta L_v$ that contain $v$,
and so  the number of $f\in \delta E$ that contain $v$
is at least
\begin{equation}\label{e:overcount}
\frac {n-d+1}d |L_v|-|L_v|-|E|. 
\end{equation}

To prove the lemma, we would like to sum this bound over
$v=v_1,v_2,\ldots,v_r$, but in this way, one $f\in\delta E$
might be counted several times. In order to avoid this,
for each $v_i$ we will count only those $f\in \delta E$ that
contain $v_i$ \emph{and avoid $v_1,\ldots,v_{i-1}$}.

In this way, from the term (\ref{e:overcount}) for $v=v_i$ we
need to subtract the number of $f\in \delta E$ that contain
both $v_i$ and some $v_j$, $j<i$. A trivial upper bound
on this number is $(i-1){n-2\choose d-1}$.
Hence
\begin{eqnarray*}
|F_{\rm hi}|&\ge& \sum_{i=1}^r\left((\tfrac{n-d+1}d-1)|L_{v_i}|- (i-1){n-2\choose d-1}
-|E|\right)\\
&\ge&
\frac nd \biggl(\sum_{i=1}^r |L_{v_i}|\biggr)- \frac{\beta^2}2 n^2
{n\choose d-1} - \beta n|E|-
O(n^d).
\end{eqnarray*}
Now $\sum_{i=1}^r |L_{v_i}|\ge |E_{\rm hi}|$. 
We finally divide by $n\choose d+1$ in order to pass to the normalized
size measure $\|.\|$, and we obtain the lemma.
\proofend

%

\heading{Dealing with low-degree vertices. }
Next, we will show that if the vertex degrees of $E$
do not exceed the average vertex degree by too much,
then $\delta E$ is even significantly larger than in
the upper bound example from Proposition~\ref{p:ubbb}.
Here it is important for the argument that we deal with triples.

In this case, we are going to count the sets of the coboundary
using two-term inclusion-exclusion, similar to the case $d=2$
in the preceding section. This leads to bounding from above
the sum of squares of the degrees of \emph{pairs} of vertices.
For this, by a suitable double counting, we use the assumption of low 
\emph{vertex} degrees, and also the fact that  the degrees
of all pairs are bounded by $\frac n2$, which follows
from the minimality of~$E$.

\begin{lemma}\label{l:low3}
Let $d=3$ and let $E\subseteq {V\choose 3}$.
Suppose that $\deg_E(p)\le \frac n2$ for
each pair $p=\{u,v\}$ of vertices, and that
$\deg_E(v)\le \sigma{n\choose 2}$ for each vertex~$v$. 
Then
$$
\|\delta E\|\ge (2-O(\sigma^{1/3}))\|E\|.
$$
\end{lemma}

\heading{Proof. } Similar to the graph case, we will count
only those $4$-tuples in $\delta E$ that contain exactly
one $e\in E$. For each $e\in E$, we count $n-d$ potential
$4$-tuples, and we subtract $1$ for each $e'\in E$ sharing
a pair with $e$. Thus,
$$
|\delta E|\ge (n-d)|E|-\sum_{p\in {V\choose2}} \deg_E(p)^2.
$$
We need to estimate the second term.

We choose a threshold parameter $\tau$, which we think
of as being much larger than $\sigma$ but still small,
and we call a pair $p$ \emph{heavy} if $\deg_E(p)\ge \tau n$,
and \emph{light} otherwise.

Each $e\in E$ shares a light pair with at most $3\tau n$ other $e'\in E$,
and so the contribution of light pairs is bounded as follows:
$$
\sum_{p \mathrm{~light}} \deg_E(p)^2 \le 3\tau n |E|.
$$

Let $E_k$ be the number of $e\in E$ with exactly $k$ heavy pairs,
$k=0,1,2,3$. Since each heavy pair has degree at most $\frac n2$,
reasoning as above, we can bound the contribution of the heavy
pairs as
$$
\sum_{p \mathrm{~heavy}} \deg_E(p)^2 \le \frac n2 (|E_1|+2|E_2|+3|E_3|).
$$
We aim at showing that $E_2\cup E_3$ is small.

Let us consider a vertex $v\in V$ and see how many 
$e\in E_2\cup E_3$ can be incident to it. More precisely,
we want to estimate $m_v$, the number of $e\in E_2\cup E_3$
that have two heavy pairs incident
to~$v$.

Let $\deg_E(v)=\sigma_v {n\choose 2}$,
where $\sigma_v\le\sigma$, and let us consider the graph $G_v:=(V,\lk(v,E))$.
The heavy pairs incident to $v$ correspond to the \emph{heavy vertices}
of $G_v$, i.e., vertices of degree at least $\tau n$, and $m_v$ is
the number of pairs in $G_v$ connecting
two heavy vertices. By simple counting, $G_v$ has at most
$\frac{\sigma_v n}\tau$ heavy vertices, and thus
$m_v\le (\frac{\sigma_v n}\tau)^2/2 \le \frac{\sigma}{\tau^2} |\lk(v,E)|$.

Summing over all vertices $v$, we have
$|E_2+E_3|\le \frac{3\sigma}{\tau^2} |E|$.
Altogether we thus have
$$
\sum_{p\in {V\choose 2}} \deg_E(p)^2\le 3\tau n |E| +
\frac n2 |E|+ O(\tfrac {\sigma}{\tau^2})n|E|.
$$
Finally, setting $\tau := \sigma^{1/3}$ and normalizing 
by ${n\choose 3}$, we obtain the claim of the lemma.
\proofend

\begin{proof}[Proof of Theorem~\ref{t:3ub}]
This is a straightforward consequence of Lemmas~\ref{l:highdeg} and~\ref{l:low3}. We consider a minimal $E\subseteq {V\choose 3}$
with $\|E\|=\alpha$. We enumerate the vertices of $V$ 
as $v_1,\ldots,v_n$ in the order of decreasing degrees.
For a suitable parameter $\beta>\alpha$ (depending on $\alpha$),
we set $r:=\beta n$, we let $E_{\rm hi}$ be those $e\in E$
that contain a vertex among $v_1,\ldots,v_r$, and let
$E_{\rm lo}:=E\setminus E_{\rm hi}$. 

We have $\sum_{i=1}^n\deg_E(v_i) = 3|E|$, and so
the degrees of the vertices $v_r,v_{r+1},\ldots,v_n$
are bounded from above by $\frac{3|E|}r$.
 Hence Lemma~\ref{l:low3} with $\sigma:= \alpha/\beta$
gives $\|\delta E_{\rm lo}\|\ge (2-(\alpha/\beta)^{1/3}) \alpha_{\rm lo}$,
where $\alpha_{\rm lo}:=\|E_{\rm lo}\|$.

Let $F_{\rm hi}$ be the set of those $f\in \delta E$ that
contain a vertex among $v_1,\ldots,v_r$; then Lemma~\ref{l:highdeg}
yields $\|F_{\rm hi}\|\ge \frac 43\alpha_{\rm hi} - O(\beta^2)$
(the $\alpha\beta$ term in the lemma is insignificant since
we assume $\alpha\le\beta$).

We now observe that if some $f\in\delta E_{\rm lo}$ does not
contain any of $v_1,\ldots,v_r$, then it belongs to
$\delta E$ (since it cannot contain
any $e\in E_{\rm hi}$). The number of $f\in \delta E_{\rm lo}$
that do contain some vertex among $v_1,\ldots,v_r$
is bounded by $r |E_{\rm lo}|$. Altogether we thus have
\begin{eqnarray*}
\|\delta E\| &\ge & \|F_{\rm hi}\|+\|\delta E_{\rm lo}\|-O(\beta\alpha_{\rm lo})\\
&\ge &\frac 43\alpha_{\rm hi} - O(\beta^2) 
+\left(2-(\alpha/\beta)^{1/3}\right) \alpha_{\rm lo}\\
&\ge& \frac 43\alpha -O(\beta^2) + \left(\frac 23 -
O((\alpha/\beta)^{1/3})\right)\alpha_{\rm lo}.
\end{eqnarray*}
If we set $\beta := C\alpha$ for a sufficiently large constant $C$,
then the term $O((\alpha/\beta)^{1/3})$ becomes smaller than
$\frac23$, and the whole term involving 
$\alpha_{\rm lo}$ is nonnegative.
Thus, we are left with $\|\delta E\|\ge \frac 43\alpha -O(\alpha^2)$
as claimed.
\end{proof}

\section{Pagodas and a Better Bound On  $\boldsymbol{c_3}$}
\label{sec:c3}

We recall the lower bound for the B\'ar\'any constant $c_d$ from
from Proposition~\ref{prop:GromovFillingBaranyConstant}:
\begin{equation}\label{e:grr}
c_d\ge \varphi_{d}(\tfrac12 \varphi_{d-1}(\tfrac13 \varphi_{d-2}(\ldots \tfrac 1d \varphi_1(\tfrac1{d+1})\ldots))),
\end{equation}
As a case study, we will concentrate on $c_3$,
the first open case. The various numerical bounds are as follows.
\begin{itemize}
\item Gromov's bound, obtained from (\ref{e:grr}) via
the basic cofilling bound and the precise value of $\varphi_1$,
is 
$$
c_3\ge \frac1{16}= 0.0625.
$$
\item For comparison, the best lower achieved by other methods,
due to Basit et al.~\cite{Basit-al}, is
$$c_3\ge 0.05448.$$
\item With \emph{maximal optimism}, assuming that the upper bounds
of Proposition~\ref{p:ubbb} are tight for $\varphi_2$ 
and $\varphi_3$, (\ref{e:grr}) would give
$$
c_3\conjge  0.0877695.
$$
\item However, the best upper bound, which we suspect to be the
truth, only gives
$$
c_3\le 4!/4^4=\tfrac 3{32}= 0.09375.
$$
\end{itemize}

Here we will show how the lower bound for $c_3$ can be improved
\emph{beyond} (\ref{e:grr}). The specific number we achieve is not
very impressive: $c_3\ge 0.06332$. However, it is important
that the proof relies only on the basic cofilling bounds on
$\varphi_2$ and $\varphi_3$. \emph{If} better lower bounds
on $\varphi_2$ or $\varphi_3$ could be proved in suitable ranges
of $\alpha$, which would improve the lower bound (\ref{e:grr}),
we would automatically get a further (slight) improvement 
from the proof below; in this sense, the method is ``orthogonal''
to bounds on the cofilling profiles.

We begin by returning to the argument in Section~\ref{s:coning}
that proves (\ref{e:grr}), and for simplicity, we specialize to the $d=3$ case.
In that argument, we considered a tetrahedron $wxyz$ in the triangulation
$\mathcal{T}$, with the corresponding set systems (cochains) 
$F_A\in {V\choose 5-|A|}$  for all nonempty sets $A\subseteq S:=\{w,x,y,z\}$.
We have  $\|F_A\|=o(1)$ unless $|A|=1$.
We also produced the set systems $F_{Ao}\in {V\choose 4-|A|}$ 
satisfying the relations
\begin{equation}\label{e:eqq1}
\delta F_{Ao}=F_A+\sum_B F_{Bo},
\end{equation}
where the sum is over all $B\subset A$ of size $|A|-1$.
Moreover, the crux of the argument was the existence
of a tetrahedron $wxyz$ for which we also have
\begin{equation}\label{e:eqq2}
F_S+\sum_{B\in {S\choose 3}} F_{Bo}=V.
\end{equation}

We introduce the notation $X\approx Y$ for sets $X,Y$ of
$k$-tuples, meaning that $|X+Y|=o(n^{k})$, or in other words, $\|X+Y\|=o(1)$.
We eliminate the sets $F_A$ with $|A|\ne 1$
from our considerations, since they are all small;
then (\ref{e:eqq1}) and (\ref{e:eqq2}) become $\approx$
relations among the various $F_{Ao}$. 

We introduce a definition reflecting these relations;
we chose to call the resulting object, a structure
made of cochains of various dimensions,  a \emph{pagoda}.
In order to make the notation more intuitive, 
we distinguish sets of various cardinalities
by different letters (with indices), writing $V$ for
sets of vertices (0-cochains), $E$ for sets of edges (1-cochains),
$F$ for sets of triples (2-cochains), and $G$ for sets of fourtuples.
We also change the
indexing of the sets into an (isomorphic but) more convenient one.
This leads to the following definition.

\begin{sloppypar}
\begin{definition}
A ($3$-dimensional) \emph{pagoda} over $V$ consists of
vertex sets $V_1,V_2,V_3,V_4\subseteq V$, edge sets
$E_{12},E_{13},\ldots,E_{34}\subseteq{V\choose 2}$,
sets $F_{123}$, $F_{124}$, $F_{134}$, $F_{234} \subseteq{V\choose 3}$ 
of triples, and a set $G=G_{1234}\subseteq {V\choose 4}$ of $4$-tuples
(the \emph{top} of the pagoda). 
The sets $V_i$, $E_{ij}$ and $F_{ijk}$ are \emph{minimal} 
and they satisfy the following relations 
(here $i,j,k$ denote mutually distinct indices):
$$
V_1+V_2+V_3+V_4\approx V,\ 
\delta V_i\approx \sum_j E_{ij},\ 
\delta E_{ij} \approx \sum_k F_{ijk},\ 
\delta F_{ijk}\approx G.
$$
\end{definition}
\end{sloppypar}

Thus, the argument of Proposition~\ref{prop:GromovFillingBaranyConstant}
shows that $c_3\ge \ctop_3\ge \lim\inf_{|V|\to\infty}\min \|G\|$,
 where the minimum is over tops $G$ of pagodas over~$V$.

We know of \emph{no example} of a pagoda
whose top is smaller than the best known upper bound for $c_3$,
i.e., $\|G\|=\tfrac 3{32}$.
So it is possible that the value of $c_3$ can be determined precisely
using a combinatorial analysis of pagodas. Unfortunately,
we have only a much weaker result.

\begin{prop}\label{p:c3bitbetter} The top $G$ of every pagoda
satisfies, for all $n$ sufficiently large, $\|G\|\ge\frac 1{16}+\eps_0$,
for a positive constant $\eps_0> 0.00082$. Consequently,
$c_3\ge \ctop_3\ge  \frac 1{16}+\eps_0> 0.06332$.
\end{prop}

A similar, but 
more complicated, argument probably also works for higher-dimensional
pagodas.

\heading{Remark. }
%
Developing the ideas from the forthcoming
proof of Proposition~\ref{p:c3bitbetter}, one can set up a rather
complicated optimization problem, whose optimum provides
a lower bound for $\ctop_3$. However, solving this (non-convex) 
optimization problem rigorously
seems rather difficult. Numerical computations
indicate that the optimum is approximately
$0.0703125$. This would be an improvement much more significant
than the one in Proposition~\ref{p:c3bitbetter}, but still
far from the suspected true value of~$c_3$.

\heading{Proof of Proposition~\ref{p:c3bitbetter}. }
We being by an outline of the argument.
We consider a pagoda 
with $\|G\|\le \frac 1{16}+\eps_0$, where $\eps_0\ge 0$
is a yet unspecified
(small) parameter. Reasoning essentially as in the
derivation of  (\ref{e:grr}) and
using the basic bound for $\varphi_3$ and $\varphi_2$
and the true value of $\varphi_1$, we find that
no $\|E_{ij}\|$ and no $\|\delta E_{ij}\|$
may be significantly larger than $\frac 18$, and
no $V_i$ may occupy much  more than $\frac14$
of the vertex set. Hence the $V_i$ are almost disjoint
and have sizes close to $\frac14$. Then one can argue that almost
all edges of $E_{12}$, say, have to go between $V_1$ and $V_2$,
as in the following picture:
\immfig{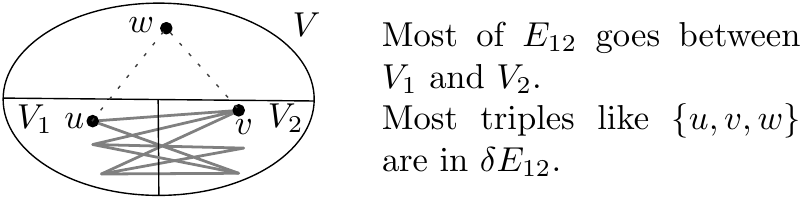}
Then, however, the coboundary $\delta E_{12}$ contains almost all
triples of the form indicated in the picture, and thus
$\|\delta E_{12}\|$ is close to $\frac3{16}$, rather than
to $\frac18$, which is a contradiction showing that
$\eps_0$ cannot be taken arbitrarily small.

Now we proceed with a more detailed and more quantitative argument.
Since $\|G\|\le \frac1{16}+\eps_0$ and $G\approx\delta F_{ijk}$,
 we have $\|F_{ijk}\|\le \frac 1{16}+\eps_0$ according to the basic
bound for $\varphi_3$ (ignoring, for simplicity, the $o(1)$ terms
coming from $\approx$). Further, since each $\delta E_{ij}$
is a sum of two $F_{ijk}$'s, the basic bound for $\varphi_2$
gives 
$$
\|E_{ij}\|\le \tfrac 18+2\eps_0,\ \ \ 1\le i<j\le 4.
$$
Similarly, since each $\delta V_i$ is the sum of three
$E_{ij}$'s, and since $\varphi_1(\alpha)=2\alpha(1-\alpha)$,
we get
\begin{equation}\label{e:Viub}
\|V_i\|\le \tfrac 14+\eps_1,
\end{equation}
where $\eps_1=\frac14\left(1-\sqrt{1-48\eps_0}\right)=
6\eps_0+72\eps_0^2+O(\eps_0^3)$ 
is given by the equation $\varphi_1(\frac14+\eps_1)=3(\frac 18+2\eps_0)$.

Next, let $E^{\rm dbl}_{ij}$
be the set of edges $e\in E_{ij}$ that belong to both 
$\delta V_i$ and $\delta V_j$, and let $E^{\rm sgl}_{ij}$
consist of those $e\in E_{ij}$ that belong to exactly
one of $\delta V_i,\delta V_j$. 

Using $\delta V_i\approx \sum_{j} E_{ij}$ and summing
over $i=1,\ldots,4$, we have
\begin{equation}\label{e:sgldbl}
\sum_{i=1}^4\|\delta V_i\|\le 2\sum_{i<j}\|E^{\rm dbl}_{ij}\|
+ \sum_{i<j}\|E^{\rm sgl}_{ij}\|.
\end{equation}
For every $i$, we have, using (\ref{e:Viub}),
$$
\|V_i\|\ge 1-\sum_{j\ne i}\|V_j\|\ge
1-3(\tfrac 14+\eps_1)\ge \tfrac 14-3\eps_1.
$$
Hence the left-hand side of (\ref{e:sgldbl}) is at least
$4\varphi_1(\frac14-3\eps_1)=4(\frac 38-\eps_2)$,
where $\eps_2$   
is defined by the last equality (we get
$\eps_2=18\eps_0+864\eps_0^2+O(\eps_0^3)$).

Further, for every $i<j$ we have $\|E^{\rm dbl}_{ij}\|+\|E^{\rm sgl}_{ij}\|
\le \|E_{ij}\|\le\frac 18+2\eps_0$. Summing this inequality over 
the six pairs $i,j$ and subtracting the result from (\ref{e:sgldbl}), we arrive at
$$
4(\tfrac 38-\eps_2)-\tfrac 68-6\cdot 2\eps_0\le \sum_{i<j}\|E^{\rm dbl}_{ij}\|.
$$
Hence there are $i$ and $j$ with 
$$\|E^{\rm dbl}_{ij}\|\ge\tfrac 18-
\tfrac23\eps_2-2\eps_0.
$$
 Let us fix the notation so that $i=1$ and $j=2$ have this
property.

Using $\|E^{\rm dbl}_{12}\|+\|E^{\rm sgl}_{12}\|\le \frac 18+2\eps_0$
again, we further obtain
\begin{equation}\label{e:Esgl}
\|E^{\rm sgl}_{12}\|\le 4\eps_0+\tfrac23\eps_2.
\end{equation}

Let us divide the edges in $E^{\rm dbl}_{12}$ into two subsets
$A$ and $B$, where $A$ consists of the edges $e\in E^{\rm dbl}_{12}$
that have one endpoint in $V_1\setminus V_2$ and the
other in $V_2\setminus V_1$. Then each edge in $B:=E^{\rm dbl}_{12}
\setminus A$ necessarily connects a vertex of
$I:=V_1\cap V_2$ to a vertex in $V\setminus (V_1\cup V_2)$.

The plan is now to show that, for $\eps_0$ sufficiently small,
the edges of $A$ ``contribute'' many triples to $\delta E_{12}$.
First we check that $B$ is small, and we begin by bounding 
the size of $I$: We have $V=V_1\cup\cdots\cup V_4$,
by inclusion-exclusion we get $1\le \|V_1\|+\|V_2\|-\|I\|+
\|V_3\|+\|V_4\|$, and thus $\|I\|\le \sum_{i=1}^4\|V_i\|-1
\le 4\eps_1$. By the minimality of $E_{12}$, each vertex in $I$
has degree at most $\frac n2$, and so 
$\|B\|\le \|I\|\le 4\eps_1$. Therefore,
\begin{equation}\label{e:Abound}
\|A\|\ge \|E^{\rm dbl}_{12}\|-\|B\|\ge \tfrac18-\tfrac 23\eps_2-2\eps_0-4\eps_1.
\end{equation}

Now let us consider an edge $e=\{u,v\}\in A$ and a vertex
$w\in V\setminus (V_1\cup V_2)$. The triple $\{u,v,w\}$ belongs
to $\delta E_{12}$ unless one of the edges $\{u,w\}$ and 
$\{v,w\}$ lies in $E_{12}$; see the following illustration:
\immfig{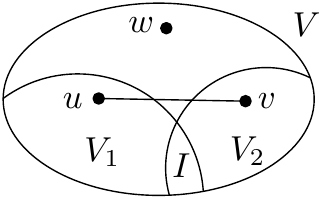}
 The (normalized) number of
``candidate triples''  $\{u,v,w\}$ with $\{u,v\}\in A$
and $w\in V\setminus (V_1\cup V_2)$ is 
$3\|A\|\cdot \|V\setminus (V_1\cup V_2)\|$. Moreover, if
one of the edges $\{u,w\},\{v,w\}$ lies in $E_{12}$, then
it lies in $E^{\rm sgl}_{12}$. At the same time,
a given edge $e'\in E^{\rm sgl}_{12}$ may ``kill''
at most $\max(|V_1\setminus V_2|,|V_2\setminus V_1|)\le
(\frac 14+\eps_1)n$ candidate triples $\{u,v,w\}$.

Hence
$$
\|\delta E_{12}\|\ge 3\|A\|(\tfrac 12-2\eps_1) - 3\|E^{\rm sgl}_{12}\|(\tfrac 
14+\eps_1)\ge \tfrac 3{16}-f(\eps_0),
$$
where, employing (\ref{e:Esgl}) and (\ref{e:Abound}), we calculate that
 $f(\eps_0)= 6\eps_0+\tfrac{27}4\eps_1-24\eps_1^2+\tfrac 32\eps_2-
2\eps_1\eps_2=\tfrac{147}2\eps_0+702\eps_0^2+O(\eps_0^3)$.
Since we have assumed $\|\delta E_{12}\|\le\tfrac 18+2\eps_0$,
we finally obtain $f(\eps_0)+2\eps_0\ge\tfrac1{16}$. 
The numerical lower bound in the proposition follows from this.
\proofend

\bibliographystyle{alpha}
\bibliography{g}

\end{document}